\documentclass[12pt]{amsart}
\usepackage[margin=1in]{geometry}
\usepackage[utf8]{inputenc}
\usepackage{amsfonts,amsmath,amsthm,amssymb}
\usepackage[unicode]{hyperref}
\usepackage{listings}
\usepackage{lastpage}
\usepackage{tikz}
\usepackage[justification=centering]{caption}
\usetikzlibrary{cd}
\usepackage[mathscr]{euscript}
\usepackage{xcolor}
\usepackage{graphicx, subcaption}
\usepackage{mathabx}
\usepackage{enumitem}

\graphicspath{{./pics/}}	

\newcommand{\C}{{\mathbb C}}
\newcommand{\W}{\mathcal{W}}
\newcommand{\I}{{\mathbb I}}
\newcommand{\D}{{\mathbb D}}
\renewcommand{\H}{{\mathbb H}}
\newcommand{\Hurw}{{\mathcal H}}
\newcommand{\Par}{\mathrm{Par}}
\newcommand{\Z}{{\mathbb Z}}
\newcommand{\R}{{\mathbb R}}

\newcommand{\N}{{\mathbb N}}
\renewcommand{\D}{{\mathbb D}}

\newcommand{\M}{\mathcal{M}}
\renewcommand{\L}{\mathcal{L}}
\newcommand{\Homeo}{\operatorname{Homeo}}

\newcommand{\id}{\operatorname{id}}

\renewcommand{\Re}{\operatorname{Re}}
\renewcommand{\Im}{\operatorname{Im}}

\renewcommand{\mod}{\mathrm{mod}}
\newcommand{\dto}{\dashrightarrow}
\newcommand{\rto}{\righttoleftarrow}
\newcommand{\T}{\mathcal{T}}

\newcommand{\hide}[1]{}

\newcounter{main}

\theoremstyle{plain}
        \newtheorem{theorem}{Theorem}[section]
        \newtheorem*{theorem*}{Theorem}
        \newtheorem*{conj*}{Conjecture}
        \newtheorem{lemma}[theorem]{Lemma}
        \newtheorem{corollary}[theorem]{Corollary}
        \newtheorem{proposition}[theorem]{Proposition}
        \newtheorem{maintheorem}[main]{Main Theorem}        

\theoremstyle{definition}
        \newtheorem{definition}[theorem]{Definition}
        \newtheorem*{definition*}{Definition}

\theoremstyle{remark}
        \newtheorem{remark}[theorem]{Remark}

        \newtheorem{example}[theorem]{Example}

        \newtheorem*{example*}{Example}
        \newtheorem*{examples*}{Examples}        
        \newtheorem*{claim}{Claim}
        \newtheorem*{claim1}{Claim 1}
        \newtheorem*{claim2}{Claim 2}
        \newtheorem*{claim3}{Claim 3}
        \newtheorem*{claim4}{Claim 4}

\newenvironment{subproof}[1][\proofname]{%
  \begin{proof}[#1]%
}{%
  \end{proof}%
}

\title[Finite and infinite degree Thurston maps with a small postsingular set]{Finite and infinite degree Thurston maps with a small postsingular set}

\author{Nikolai Prochorov}
\address {Aix-Marseille Université, CNRS, Institut de Mathématiques de Marseille, 13003 Marseille, France}
\email{nikolai.prochorov@etu.univ-amu.fr, prochorov41@gmail.com}

\date{\today}


\begin{document}

\begin{abstract}
    We develop the theory of Thurston maps that are defined everywhere on the topological sphere $S^2$ with a possible exception of a single essential singularity. We establish an analog of the celebrated W.~Thurston's characterization theorem for a broad class of such Thurston maps having four postsingular values. To achieve this, we analyze the corresponding pullback maps defined on the one-complex dimensional Teichm\"uller space. This analysis also allows us to derive various properties of Hurwitz classes of the corresponding Thurston~maps.
    

\end{abstract}

\maketitle

\tableofcontents

\newpage

\section{Introduction}

\subsection{Thurston theory for finite and infinite degree maps}\label{subsec: inro thurston theory}
In the one-dimensional rational dynamics, the crucial role is played by the family of \textit{postcritically finite} (or \textit{pcf} in short) rational maps, i.e., maps with all critical points being periodic or pre-periodic. In this context, one of the most influential ideas has been to abstract from the rigid underlying complex structure and consider the more general setup of postcritically finite \emph{branched self-coverings} of the topological $2$-sphere $S^2$. Nowadays, orientation-preserving pcf branched covering maps $f\colon S^2 \to S^2$ of topological degree $\deg(f)\geq 2$ are called \emph{Thurston maps} (\textit{of finite degree}), in honor of William Thurston, who introduced them to deepen the understanding of the dynamics of postcritically finite rational maps~on~$\widehat{\C}$.

These ideas can be extended to the transcendental setting to explore the dynamics of \textit{postsingularly finite} (\textit{psf} in short) meromorphic maps. A meromorphic map is called \textit{postsingularly finite} if it has finitely many \textit{singular values}, and each of them eventually becomes periodic or lands on the essential singularity under the iteration. We can generalize the notion of a Thurston map to include postsingularly finite \textit{topologically holomorphic} non-injective maps $f \colon X \to S^2$, where $X$ is a punctured topological sphere, $f$ does not extend continuously to the entire $S^2$ and meets a technical condition of being a \textit{parabolic type} map; see Sections~\ref{subsec: topologically holomorphic maps} and~\ref{subsec: thurston maps}. Note that in this case the map $f$ must be \textit{transcendental}, meaning that it has infinite topological degree. For simplicity, we will use the notation $f \colon S^2 \dto S^2$ to indicate that the Thurston map $f$, whether finite or infinite degree, might not be defined at a single~point~of~$S^2$.

For a Thurston map $f \colon S^2 \dto S^2$, the \textit{postsingular set} $P_f$ is defined as the union of all orbits of the singular values of the map $f$.  It is important to note that some of these orbits might terminate after several iterations, if a singular value reaches the point where the map~$f$ is not defined. The elements of the postsingular set $P_f$ are called the \textit{postsingular values} of the Thurston map $f$. If the map $f$ is defined on the entire sphere $S^2$, we simply refer to its \textit{postcritical set} and \textit{postcritical values}, as the set of singular values of $f$ coincides with the set of its \textit{critical values}.  Two Thurston maps are called \emph{combinatorially} (or \emph{Thurston}) \emph{equivalent} if they are conjugate up to isotopy relative to their postsingular~sets; see Definition \ref{def: comb equiv}.

A fundamental question in this context is whether a given Thurston map $f$ can be \emph{realized} by a psf meromorphic map with the same combinatorics, that is, if $f$ is combinatorially equivalent to a psf meromorphic map. If the Thurston map $f$ is not realized, then we say that $f$ is \textit{obstructed}. William Thurston answered this question for Thurston maps of finite degree in his celebrated \emph{characterization of rational maps}: 
if a finite degree Thurston map $f \colon S^2 \to S^2$ has a \emph{hyperbolic orbifold} (this is always true, except for some well-understood special maps), then $f$ is realized by a pcf rational map if and only if $f$ has no \emph{Thurston obstruction} \cite{DH_Th_char}. Such an obstruction is given by a finite collection of disjoint simple closed curves in $S^2 - P_f$ with certain invariance properties under the map $f$. In many instances, it suffices to restrict to simpler types of Thurston obstructions provided by \textit{Levy cycles} or even \textit{Levy fixed curves}; see Definition \ref{def: levy cycle}, and \cite[Theorem 10.3.8]{Hubbard_Book_2}, \cite[Corollary~1.5]{critically_fixed}, or \cite[Theorems 7.6 and 8.6]{park} for examples of such cases.

The same characterization question can be also asked in the transcendental setting. The first breakthrough in this area was obtained in \cite{HSS}, where it was shown that an \textit{exponential} Thurston map is realized if and only if it has no Levy cycle. In this context, the exponential Thurston map is defined as a Thurston map with two singular values, both of which are \textit{omitted}, and one of which is the only essential singularity of the considered Thurston map. Furthermore, the results of \cite{Shemyakov_Thesis} and \cite{our_approx} suggest that a Thurston-like criterion for realizability may hold in a greater generality. However, the characterization question in the transcendental setting remains largely open, as many of the techniques used in Thurston theory for finite degree maps do not extend to this context.

Thurston theory lays out the relationship between the topological properties of a map, its dynamics, and its geometry in terms of the existence of a holomorphic realization. Further- more, it is strongly connected with the combinatorial and algebraic aspects of the dynamics of pcf rational and psf meromorphic maps. The results mentioned above have substantial applications for both rational and transcendental dynamics. For instance, Thurston's characterization result has allowed to classify various families of postcritically finite rational maps or finite degree Thurston maps in terms of combinatorial models \cite{Poirier_Thesis, Poirier, LifitingTrees}, \cite{Newton, RussellDierk_Class}, \cite{H_Tischler, critically_fixed}.  Building on the result of \cite{HSS}, similar classifications were obtained in \cite{class_of_exp} in terms of \textit{kneading sequences} and in \cite{hubbard_trees_for_exp} in terms of \textit{homotopic Hubbard treed} for the family of postsingularly finite exponential maps. Moreover, the concept of a homotopic Hubbard tree was extended to general postsingularly finite entire maps in \cite{Pfrang_thesis} (see also \cite{dreadlock}). It is plausible that a Thurston-like criterion is the final missing ingredient for the complete classification of the family of all psf entire maps. 

\subsection{Pullback maps}\label{subsec: intro pullback maps}
The key method in determining whether a given finite or infinite degree Thurston map $f \colon S^2 \dto S^2$ with the postsingular set $P_f = A$ is realized by a psf meromorphic map is the analysis of the dynamics of a holomorphic operator~$\sigma_f$, known as the \textit{pullback map}, defined on a complex manifold called the \textit{Teichm\"uller space}~$\T_A$; see Sections~\ref{subsec: teichmuller spaces} and~\ref{subsec: pullback maps}.  Crucially, the Thurston map $f$ is realized if and only if the pullback map $\sigma_f$ has a fixed point in $\T_A$; see \ref{prop: fixed point of sigma}. Moreover, the dynamics of the pullback map encodes many other properties of the corresponding Thurston map; see, for instance, \cite{pullback_map}, \cite{Pullback_invariants}, and \cite{correpondences}.

Instead of working directly in the Teichm\"uller space $\T_A$, it is often more convenient to work in a simpler complex manifold $\M_A$, known as a \textit{moduli space}. This space, roughly speaking, encodes all possible complex structures biholomorhic to a punctured Riemann sphere that can be put on the punctured topological sphere $S^2 - A$ (see Section \ref{subsec: teichmuller spaces} for the precise definition). There is a natural projection map $\pi \colon \T_A \to \M_A$ that is a holomorphic universal covering. However, the map $\sigma_f$ rarely descends to a map on the moduli space~$\M_A$. Nevertheless, for a finite degree Thurston map $f \colon S^2 \to S^2$, there exists a complex manifold $\W_f$, known as the \textit{Hurwitz space} of the Thurston map $f$, along with the holomorphic \textit{$X$-map} $X_f \colon \W_f \to \M_A$, a holomorphic covering map $Y_f \colon \W_f \to \M_A$, called the \textit{$Y$-map}, and a holomorphic covering $\omega_f \colon \T_A \to \W_f$, such that the following~diagram~commutes~\cite[Section 2]{critically_finite_endomorphisms}:
\begin{figure}[h]
        \begin{tikzcd}
            \T_A \arrow[rr, "\sigma_f"] \arrow[dd, "\pi"] \arrow[dr, "\omega_f"] & & \T_A \arrow [dd, "\pi"]\\
            & \W_f \arrow[dl, "Y_f"] \arrow[dr, "X_f"] & \\
            \M_A  & & \M_A            \end{tikzcd} 
        \caption{Fundamental diagram.}\label{fig: fund diag}
        \end{figure}
        
In other words, the pullback map $\sigma_f$ is semi-conjugate to a self-correspondence $X_f \circ Y_f^{-1}$ of the moduli space $\M_A$. If $f \colon S^2 \to S^2$ with $P_f = A$ is a finite degree Thurston map, then the $Y$-map $Y_f \colon \W_f \to \M_A$ has a finite topological degree; see \cite[Theorem 2.6]{critically_finite_endomorphisms}. This observation plays a crucial role in the proof of Thurston's characterization of rational maps. In fact, it allows to conclude that for a Thurston map $f$ with a hyperbolic orbifold, the $\sigma_f$-orbit $(\sigma_f^{\circ n}(\tau))$ of $\tau \in \T_A$ converges (indicating that $f$ is realized), if the projection $(\pi(\sigma_f^{\circ n}(\tau)))$ of this orbit visits some compact set of the moduli space $\M_A$ infinitely many times; see \cite[Section~10.9 and Lemma 10.11.9]{Hubbard_Book_2} and \cite[Proof of Theorem~2.3,~p.~20]{Nikita}.

The objects introduced above, along with commutative diagram \eqref{fig: fund diag} and the fact that $Y_f \colon \W_f \to \M_A$ has a finite degree, have broad applications beyond the proof of Thurston's characterization of rational maps; see, for example, \cite{BarNekr_Twist, pullback_map, Nikita, Nikita_canonical, critically_finite_endomorphisms, Lodge_Boundary, Pullback_invariants, Origami, smith_thesis, smith, correpondences}. These tools, for instance, allow to simultaneously study the entire \textit{Hurwitz} (\textit{equivalence}) \textit{class} $\Hurw_f$ of the finite degree Thurston map $f$. Here, two Thurston maps $f_1\colon S^2 \dto S^2$ and $f_2\colon S^2 \dto S^2$ with $P_{f_1} = P_{f_2}$ are said to be \textit{Hurwitz equivalent} if there exist orientation-preserving homeomorphisms $\phi_1, \phi_2 \colon S^2 \to S^2$ such that $\phi_1 | P_{f_1} = \phi_2 | P_{f_1}$ and $\phi_1 \circ f_1 = f_2 \circ \phi_2$; see \cite{BarNekr_Twist, Lodge_Boundary, Pullback_invariants} for examples of results on Hurwitz~classes. For instance, one can pose a question whether a given Thurston map is \textit{totally unobstructed}, i.e., $\Hurw_f$ consists of only realized Thurston maps.

Commutative diagram \eqref{fig: fund diag} is particularly powerful in two specific cases. The first is when the Thurston map $f \colon S^2 \to S^2$ has the postcritical set $P_f = A$ consisting of exactly four points. In this situation, the spaces $\T_A$, $\M_A$, and $\W_f$ are simply Riemann surfaces. In fact, the Teichm\"uller space is biholomorphic to the unit disk $\D$ and the moduli space $\M_A$ is biholomorphic to the three punctured Riemann sphere $\Sigma = \widehat{\C} - \{0, 1, \infty\}$. This allows the use of powerful machinery of one-dimensional holomorphic dynamics to study pullback maps. For example, this approach was utilized in \cite{smith_thesis} to derive an alternative proof of Thurston's characterization of rational maps in the case of four postcritical values, as well as in \cite{smith_thesis, smith} to investigate 
the \textit{global curve attractor conjecture}, which was ultimately resolved in \cite{correpondences} for all pcf rational maps with four postcritical values. Secondly, when the $X$-map $X_f$ is injective, the ``inverse'' of $\sigma_f$ descends to the so-called \textit{$g$-map} $g_f := Y_f \circ X_f^{-1}$, which is defined on the subset $X_f(\W_f)$ of the moduli space $\M_A$; see \cite[Section 5]{critically_finite_endomorphisms} for examples of finite degree Thurston maps that satisfy this condition.

The theory of moduli maps, as outlined above, is developed for finite degree Thurston maps but has not yet been established in the context of transcendental Thurston maps. In this paper, we consider a certain family of Thurston maps, that includes maps of both finite and infinite degree, with four postsingular values and show that the corresponding pullback maps admit an analogue of commutative diagram \eqref{fig: fund diag}, where the $X$-map is always injective, but the $Y$-map has infinite degree if the initial Thurston map $f \colon S^2 \dto S^2$ is transcendental. Using tools of one-dimensional holomorphic dynamics and hyperbolic geometry, we establish a Thurston-like realizability criterion for this family of maps. In particular, we demonstrate that the obstacle of the $Y$-map having infinite degree can be finessed. Afterward, we illustrate how the developed machinery can be used to investigate the properties of the corresponding Hurwitz classes.

\subsection{Main results}\label{subsec: main results}
In this paper, we study the family of Thurston maps $f \colon S^2 \dto S^2$ satisfying the following conditions:
\begin{enumerate}[label=(\Alph*)]
    \item \label{it: small s set} the map $f$ has at most three singular values;

    \item \label{it: small ps set} the postsingular set $P_f$ consists of exactly four points;

    \item \label{it: extra property} there exists a set $B \subset P_f$ such that $|B| = 3$, $S_f \subset B$, and $|\overline{f^{-1}(B)} \cap P_f| = 3$.
\end{enumerate}

Here, $\overline{f^{-1}(B)}$ is the closure of the set $f^{-1}(B)$ in the topology of $S^2$. In particular, it coincides with $f^{-1}(B)$ if the Thurston map $f$ has a finite degree; otherwise, it also includes the point where $f$ is not defined.

Clearly, conditions \ref{it: small ps set} and \ref{it: extra property} are independent from the function-theoretical properties of the map $f$. More specifically, if the map $f$ has at most three singular values, then these conditions can be verified by analyzing the dynamics of the map $f$ on the finite set~$P_f$. For instance, in the case of \textit{entire} Thurston maps (those that can be restricted to self-maps of $\R^2$; see Section \ref{subsec: topologically holomorphic maps}) with three singular and four postsingular values, three out of seven possible \textit{postsingular portraits} satisfy condition \ref{it: extra property}; see Example \ref{ex: entire}. More examples of families of Thurston maps that meet these conditions can be found in Section \ref{subsec: examples}. 

Although conditions \ref{it: small s set}--\ref{it: extra property} are quite restrictive, there are still uncountably many pairwise combinatorially inequivalent both realized and obstructed Thurston maps that meet them; see Remark~\ref{rem: uncountably many}. Notably, these conditions are preserved under Hurwitz equivalence of Thurston~maps. Finally, it is worth mentioning that all of our further results work in a slightly more
general setting of \textit{marked} Thurston maps satisfying analogous properties to \ref{it: small s set}--\ref{it: extra property}; see~Sections~\ref{subsec: thurston maps}~and~\ref{sec: thurston theory}.

\subsubsection{Characterization problem} \label{subsubsec: char} We establish an analog of Thurston's characterization result for the family of Thurston maps that satisfy conditions \ref{it: small s set}--\ref{it: extra property}. In fact, we show that it is sufficient to consider one of the simplest types of Thurston obstructions --~Levy fixed curves~-- to determine whether such a Thurston map is realized. For a Thurston map $f \colon S^2 \dto S^2$, a Levy fixed curve is a simple closed curve $\gamma\subset S^2 - P_f$ such that $\gamma$ is \textit{essential}, i.e., it cannot be shrinked to a point by a homotopy in $S^2 - P_f$, and there exists another simple closed curve  $\widetilde{\gamma} \subset f^{-1}(\gamma)$ such that $\gamma$ and $\widetilde{\gamma}$ are homotopic in $S^2 - P_f$ and $f|\widetilde{\gamma}\colon \widetilde{\gamma}\to \gamma$ is a homeomorphism. If the map $f$ is injective on one of the connected components of $S^2 - \widetilde{\gamma}$, then we say that the Levy fixed curve $\gamma$ is \textit{weakly degenerate}.

\begin{maintheorem}\label{mainthm A}
    Let $f \colon S^2 \dto S^2$ be a Thurston map of finite or infinite degree that satisfies conditions \ref{it: small s set}--\ref{it: extra property}. Then $f$ is realized if and only if it has no weakly degenerate Levy fixed curve. Moreover, if $f$ is obstructed, then it has a unique Levy fixed curve up to homotopy in $S^2 - P_f$; otherwise, it is realized by a psf meromorphic map that is unique up to M\"obuis conjugation.
\end{maintheorem}

To prove Main Theorem \ref{mainthm A}, we start by showing that the corresponding pullback map $\sigma_f$ defined on the Teichm\"uller space $\T_A \sim \D$, where $A = P_f$, admits an analog of commutative diagram \eqref{fig: fund diag}, where the $X$-map is injective, the $Y$-map is a covering, potentially of infinite degree, and the analog of the Hurwitz space $\W_f$ is a finitely or countably punctured Riemann sphere (see Proposition~\ref{prop: 1-dim pullback maps} and Remark \ref{rem: comparision with finite degree case}). Further analysis reveals a crucial observation similar to that in the proof of Thurston's characterization result: if the sequence $(\pi(\sigma_f^{\circ n}(\tau)))$ with $\tau \in \T_A$ visits a certain compact set of the moduli space $\M_A$ infinitely many times, then the $\sigma_f$-orbit $(\sigma_f^{\circ n}(\tau))$ of $\tau$ converges to the unique fixed point of $\sigma_f$ (see Claim 1 of the proof~of~Theorem~\ref{thm: iteration on unit disk}). 

Moreover, we establish a more refined result: if the pullback map $\sigma_f$ does not have a fixed point (indicating that the Thurston map $f$ is obstructed), then the sequence $(\pi(\sigma_f^{\circ n}(\tau)))$ converges to the same ``cusp'' $x \in \partial \M_A \sim \partial \Sigma = \{0, 1, \infty\}$ of the moduli space $\M_A$, regardless of the choice of $\tau \in \T_A$. Furthermore, the map $g_f = Y_f \circ X_f^{-1}$ can be holomorphically extended to a neighborhood of this cusp and $x$ becomes a repelling fixed point of $g_f$. It is worth noting that these results hold not only for pullback maps, but also in a broader class of holomorphic self-maps of the unit disk; see Theorem \ref{thm: iteration on unit disk}.
Finally, this analysis provides a sufficient control over the dynamics of $\sigma_f$ to derive the existence of a Levy fixed curve for the obstructed Thurston map $f\colon S^2 \dto S^2$.

Main Theorem \ref{mainthm A} is the first result addressing the characterization problem in the transcendental setting that is established with minimal reliance on the function-theoretical properties of the considered Thurston maps. This differs from \cite{HSS}, which focuses on exponential Thurston maps, and \cite{Shemyakov_Thesis}, which is primarily devoted to \textit{structurally finite} Thurston maps. Note that the geometric and analytic properties of entire or meromorphic maps, even those with few singular values, can be highly varied and subtle; see~\cite{qc_foldings, order_conjecture, models_for_class_S}.

Main Theorem \ref{mainthm A} also offers an alternative proof for the characterization of postsingularly finite exponential maps \cite[Theorem~2.4]{HSS} for the case of four postsingular values; see Example \ref{ex: exponential}. While the proof in \cite{HSS} relies on the intricate machinery of \textit{integrable quadratic differentials} and their \textit{thick-thin decompositions}, our approach uses more explicit techniques that shed the light on the geometry of the pullback dynamics on Teichm\"uller and moduli spaces. At the same time, Main Theorem \ref{mainthm A} provides a novel proof of Thurston's characterization of rational maps within a broad class of examples, and this proof does not rely on the fact that $Y$-map has a finite~degree.

\subsubsection{Hurwitz classes} \label{subsubsec: hurwitz} Techniques explained in Section \ref{subsubsec: char} also allow us to derive several properties of Hurwitz classes:

\begin{maintheorem}\label{mainthm B}
    Let $f \colon S^2 \dto S^2$ be a Thurston map of finite or infinite degree that satisfies conditions \ref{it: small s set}--\ref{it: extra property}. Then 
    \begin{enumerate}
        \item \label{it: tot unobstr} $f$ is totally unobstructed if and only if there are no two points $a, b \in P_f$ such that $\deg(f, a) = \deg(f, b) = 1$ and $f(\{a, b\})$ equals $\{a, b\}$ or $P_f - \{a, b\}$;

        \item if $f$ is not totally unobstructed, then its Hurwitz class $\Hurw_f$ contains infinitely many pairwise combinatorially inequivalent obstructed Thurston maps;

        \item \label{it: realiz} if $f$ has infinite degree, then its Hurwitz class $\Hurw_f$ contains infinitely many pairwise combinatorially inequivalent realized Thurston maps.
    \end{enumerate}
\end{maintheorem}
The main tool for proving Main Theorem \ref{mainthm B} is the relationship between fixed points of the map $g_f = Y_f \circ X_f^{-1}$ and the elements of the Hurwitz class $\Hurw_f$. Let $f \colon S^2 \dto S^2$ be a Thurston map with $P_f = A$ satisfying assumptions \ref{it: small s set}--\ref{it: extra property}. We show that if the corresponding map~$g_f$ can be holomorphically extended to a neighborhood of $x \in \partial \M_A \sim \partial \Sigma = \{0, 1, \infty\}$ and $x$ becomes a repelling fixed point of $g_f$, then the Hurwitz class $\Hurw_f$ must contain an obstructed Thurston map. Moreover, $f$ is totally unobstructed if and only if none of the ``cusps'' of the moduli space $\M_A$ exhibit this behavior; see Proposition \ref{prop: fixed points and obstructed maps}. Using additional properties of the map $g_f$, we establish a simple criterion, as in item \eqref{it: tot unobstr} of Main Theorem \ref{mainthm B}, for determining whether a given Thurston map with properties \ref{it: small s set}--\ref{it: extra property} is totally unobstructed. Furthermore, with the understanding of possible obstructions provided by Main Theorem \ref{mainthm A}, we can construct infinitely many pairwise combinatorially inequivalent obstructed Thurston maps within the Hurwitz class $\Hurw_f$ starting from just one of them.

To prove item \eqref{it: realiz} of Main Theorem \ref{mainthm B}, we show that a fixed point $x \in \M_A$ of the map~$g_f$ corresponds to a realized Thurston map within the Hurwitz class $\Hurw_f$; see Proposition \ref{prop: fixed points and realized maps}. Furthermore, only finitely many fixed points of $g_f$ can correspond to the same Thurston map up to combinatorial equivalence (see the proof of Theorem \ref{thm: B}). The desired result then follows because the map $g_f$ has infinitely many fixed points when $f$ is transcendental;~see~Proposition~\ref{prop: 1-dim pullback maps}~and~Lemma~\ref{lemm: inf many repelling fixed points}.

Similar connections between the fixed points of the map $g_f$ and Thurston maps within the Hurwitz class $\Hurw_f$ are already established in the context of finite degree Thurston maps (cf. \cite[Propositions 4.3 and 4.4]{critically_finite_endomorphisms} and \cite[Theorem 9.1]{Pullback_invariants}). However, their extensions to the setting of transcendental Thurston maps are novel contributions. Additionally, as an application of Main Theorem \ref{mainthm B}, we can obtain the following result regarding the structure of parameter spaces of finite-type meromorphic maps (see Definition \ref{def: parameter space}):

\begin{corollary}\label{corr: intro}
    Let $g \colon \C \to \widehat{\C}$ be a transcendental meromorphic map such that $|S_g|~\leq~3$. Then its parameter space $\Par(g)$ contains infinitely many pairwise (topologically or conformally) non-conjugate psf maps with four postsingular values.
\end{corollary}

\subsection{Organization of the paper}\label{subsec: Organization of the paper}

Our paper is organized as follows. In Section \ref{sec: background}, we review some general background. In Section \ref{subsec: notation and basic concepts}, we fix the notation and state some basic definitions. We discuss topologically holomorphic maps in Section \ref{subsec: topologically holomorphic maps}. The necessary background on Thurston maps is covered in Section \ref{subsec: thurston maps}. Section \ref{subsec: teichmuller spaces} introduces the Teichm\"uller and moduli spaces of a marked topological sphere. Finally, in Section \ref{subsec: pullback maps}, we define pullback maps, discuss their basic properties and their relations with the associated Thurston~maps.

In Section \ref{sec: hyperbolic tools}, we present several results concerning the hyperbolic geometry and dynamics of holomorphic self-maps of the unit disk. Section \ref{subsec: levy cycles} provides tools for identifying obstructions for Thurston maps with four postsingular values. We establish some estimates for hyperbolic contraction of inclusion maps between two hyperbolic Riemann surfaces in Section \ref{subsec: estimating contraction}. In Section \ref{subsec: iteration on the unit disk}, we investigate dynamics of holomorphic self-maps of the unit disk satisfying certain additional assumptions.

Further, in Section \ref{sec: thurston theory}, we develop the Thurston theory for a family of Thurston maps satisfying condition \ref{it: small s set}--\ref{it: extra property}. In particular, in Section \ref{subsec: characterization problem}, we address the characterization problem for this class of Thurston maps and prove Main Theorem \ref{mainthm A}. We study properties of Hurwitz classes and prove Main Theorem \ref{mainthm B} and Corollary \ref{corr: intro} in Section~\ref{subsec: hurwitz classes}. Finally, in Section \ref{subsec: examples}, we provide and analyze various examples.

\textbf{Acknowledgments.} I would like to express my deep gratitude to my thesis advisor, Dierk Schleicher, for introducing me to the fascinating world of Transcendental Thurston Theory. I am also profoundly thankful to Kevin Pilgrim and Lasse Rempe for the their valuable suggestions and for many helpful and inspiring discussions. I would like to thank \textit{Centre National de la Recherche Scientifique} (\textit{CNRS}) for supporting my visits to the University of Saarland and University of Liverpool, where these conversations took place. Special thanks go to Anna Jov\'e and Zachary Smith for the discussions on the dynamics of holomorphic self-maps of the unit disk and their diverse applications.

\section{Background} \label{sec: background}

\subsection{Notation and basic concepts} \label{subsec: notation and basic concepts} The sets of positive integers, non-zero integers, integers, real and complex numbers are denoted by $\N$, $\Z^*$, $\Z$, $\R$, and $\C$, respectively. We use the notation $\I:=[0,1]$ for the closed unit interval on the real line, $\D :=\{z \in \C\colon |z|<1\}$ for the
open unit disk in the complex plane, $\D^* := \D - \{0\}$ for the punctured unit disk, $\C^* := \C - \{0\}$ for the punctured complex plane, $\H := \{z \in \C: \Im(z) > 0\}$ for the upper half-plane, $\widehat{\C}:=\C\cup\{\infty\}$ for the Riemann sphere, and $\Sigma$ for the three-punctured Riemann sphere $\widehat{\C} - \{0, 1, \infty\}$. The open and closed disks of radius~$r>0$ centered at~$0$ are denoted by $\D_r$ and $\overline{\D}_r$, respectively. Finally, $\arg(z) \in [0, 2\pi)$ and $|z|$ denote the argument and the absolute value, respectively, of the complex number $z$.

We denote the oriented 2-dimensional sphere by $S^2$. In this paper, we treat it as a purely topological object. In particular, our convention is to write $g\colon \C \to \widehat{\C}$ or $g \colon \widehat{\C} \to \widehat{\C}$ to indicate that $g$ is holomorphic, and $f\colon S^2 \to S^2$ if $f$ is only continuous. The same rule applies to the notation $g \colon \widehat{\C} \dto \widehat{\C}$ and $f \colon S^2 \dto S^2$ (see Section \ref{subsec: topologically holomorphic maps} for the details).

The cardinality of a set $A$ is denoted by $|A|$ and the identity map on $A$ by $\id_A$. If $f\colon U\to V$ is a map and $W\subset U$, then $f|W$ stands for the restriction of $f$ to $W$. If $U$ is a topological space and $W\subset U$, then $\overline W$ denotes the closure and $\partial W$ the boundary of $W$ in $U$.

A subset $D$ of $\widehat{\C}$ is called an \emph{open Jordan region} if there exists an injective continuous map $\varphi\colon \overline{\D} \to \widehat{\C}$ such that $D = \varphi(\D)$. In this case, $\partial D = \varphi(\partial \D)$ is a simple closed curve in $\widehat{\C}$. A domain $U \subset \widehat{\C}$ is called an \textit{annulus} if $\widehat{\C} - U$ has two connected components. The \textit{modulus} of an annulus $U$ is denoted by $\mod(U)$ (see \cite[Proposition~3.2.1]{Hubbard_Book_1}~for~the~definition).

Let $U$ and $V$ be topological spaces. A continuous map $H\colon U\times \I \to V$ is called a \emph{homotopy} from $U$ to $V$. When $U = V$, we simply say that $H$ is a homotopy \textit{in} $U$. Given a homotopy $H \colon U \times \I \to V$, for each $t \in \I$, we associate the \emph{time-$t$ map} $H_t:=H(t,\cdot)\colon U\to V$. Sometimes it is convenient to think of the homotopy $H$ as a continuous family of its time maps~$(H_t)_{t \in \I}$. The homotopy $H$ is called an (\emph{ambient}) \emph{isotopy} if the map $H_t\colon U\to V$ is a homeomorphism for each $t\in \I$. Suppose $A$ is a subset of $U$. An isotopy $H\colon U\times \I \to V$ is said to be an isotopy \emph{relative to $A$} (abbreviated ``$H$ is an isotopy rel.\ $A$'') if $H_t(p) = H_0(p)$ for all $p\in A$ and $t\in \I$. 

Given $M, N \subset U$, we say that \emph{$M$ is homotopic} (\textit{in $U$}) to $N$ if there exists a homotopy $H\colon U\times \I \to U$  with $H_0 = \id_U$ and $H_1(M) = N$. Two homeomorphisms $\varphi_0, \varphi_1\colon U\to V$ are called \emph{isotopic} (\emph{rel.\ $A\subset U$}) if there exists an
isotopy $H\colon U\times \I \to V$ (rel.\ $A$) with $H_0=\varphi_0$ and $H_1 = \varphi_1$. 


We assume that every topological surface $X$ is oriented. We denote by $\Homeo^+(X)$ and $\Homeo^+(X, A)$ the group of all orientation-preserving self-homeomorphisms of $X$ and the group of all orientation-preserving self-homeomorphisms of $X$ fixing $A$ pointwise, respectively. We use the notation $\Homeo^+_0(X, A)$ for the subgroup of $\Homeo^+(X, A)$ consisting of all homeomorphisms isotopic rel.\ $A$ to $\id_X$.

\subsection{Topologically holomorphic maps} \label{subsec: topologically holomorphic maps}

In this section, we briefly recall the definition of a topologically holomorphic map and some of its basic properties (for more detailed discussion see \cite[Section 2.3]{our_approx}; see also \cite{Stoilow}).

\begin{definition}\label{def: top hol}
    Let $X$ and $Y$ be two connected topological surfaces. A map $ f\colon X \to Y$ is called \textit{topologically holomorphic} if it satisfies one of the following four equivalent conditions:
    \begin{enumerate}
        \item for every $p \in X$ there exist $d \in \N$, a neighborhood $U$ of~$x$, and two orientation-preserving homeomorphisms $\psi\colon U \to \D$ and $\varphi\colon f(U) \to \D$ such that $\psi(p) = \varphi(f(p)) = 0$ and $(\varphi \circ f \circ \psi^{-1})(z) = z^d$ for all $z \in \D$;

        \item $f$ is continuous, open, \textit{discrete} (i.e., $f^{-1}(q)$ is discrete in $X$ for very $q \in Y$), and for every $p \in X$ such that $f$ is locally injective at $p$, there exists a neighborhood $U$ of $p$ for which $f|U\colon U \to f(U)$ is an orientation-preserving homeomorphism;

        \item there exist Riemann surfaces $S_X$ and $S_Y$ and orientation-preserving homeomorphisms $\varphi \colon Y \to S_Y$ and $\psi \colon X \to S_X$ such that $\varphi \circ f \circ \psi^{-1} \colon S_X \to S_Y$ is a holomorphic map;

        \item \label{it: pulling back} for every orientation-preserving homeomorphism $\varphi \colon Y \to S_Y$, where $S_Y$ is a Riemann surface, there exist a Riemann surface $S_X$ and an orientation-preserving homeomorphism $\psi \colon X \to S_X$ such that $\varphi \circ f \circ \psi^{-1} \colon S_X \to S_Y$ is a holomorphic map.
    \end{enumerate}
\end{definition}

Note that in condition (\ref{it: pulling back}) of Definition \ref{def: top hol}, the homeomorphism $\psi$ is defined uniquely up to post-composition with a conformal automorphism of $S_X$ for fixed $\varphi$ and $S_X$.

It is straightforward to define the concepts of \textit{regular}, \textit{singular}, \textit{critical}, and \textit{asymptotic values}, as well as \textit{regular} and \textit{critical points} and their \textit{local degrees} (denoted by $\deg(f, \cdot)$) for the topologically holomorphic map $f$ (see \cite[Definition 2.7]{our_approx}). We denote by $S_f \subset Y$ the \textit{singular set} of~$f$, i.e., the union of all singular values of the topologically holomorphic map $f \colon X \to Y$. We say that the map $f$ is \textit{of finite type} or belongs to the \textit{Speiser class} $\mathcal{S}$ if the set $S_f$ is finite.

In this paper, we study topologically holomorphic maps $f \colon X \to S^2$, where $X$ is either the sphere $S^2$ or the punctured sphere $S^2 - \{e\}$. In the latter case, we assume that $f$ cannot be extended as a topologically holomorphic map to the entire sphere $S^2$.
For the sake of simplicity, we are going to use the notation $f \colon S^2 \dto S^2$ in order to indicate that $f$ might not be defined at a single point $e \in S^2$. Similar to the holomorphic case, the point $e$ is referred as the \textit{essential singularity} of the map $f$. Likewise, for a holomorphic map $g$ defined everywhere on~$\widehat{\C}$ with the possible exception of a single essential singularity, we use the notation $g \colon \widehat{\C} \dto \widehat{\C}$, and we say that $g \colon \widehat{\C} \dto \widehat{\C}$ is holomorphic.

It is possible to derive the following isotopy lifting property for topologically holomorphic maps as above in the case when they are of finite type (cf. \cite[Propostion 2.3]{Rempe_Epstein}).

\begin{proposition}\label{prop: isotopy lifting property}
    Let $f \colon X \to S^2$ and $\widetilde{f}\colon \widetilde{X} \to S^2$ be topologically holomorphic maps of finite type, where $X$ and $\widetilde{X}$ are either topological spheres or punctured topological spheres. Suppose that $\varphi_0 \circ f = \widetilde{f} \circ \psi_0$ for some $\varphi_0, \psi_0 \in \Homeo^+(S^2)$.
    Let $A \subset S^2$ be a finite set containing $S_{f}$ and $\varphi_1 \in \Homeo^+(S^2)$ is isotopic rel.\ $A$ to $\varphi_0$. Then $\varphi_1 \circ f = \widetilde{f} \circ \psi_1$ for some $\psi_1 \in \Homeo^+(S^2)$ isotopic rel.\ $f^{-1}(A) \cup (S^2 - X) \supset \overline{f^{-1}(A)}$ to $\psi_0$.
\end{proposition}

\begin{proof}
    Let $(\varphi_t)_{t \in \I}$ be the corresponding isotopy. From the definition of a singular value, it follows that the restrictions
    $\varphi_t \circ f | Y  \colon Y \to Z$ are covering maps for each $t \in \I$, where $Y := X - f^{-1}(A)$ and $Z := S^2 - A$. Therefore, \cite[Lemma 2.7]{Astorg_Benini_Fagella} implies the existence of an isotopy $(\phi_t)_{t \in \I}$ in~$Y$ such that $\varphi_t \circ f = \varphi_0 \circ f \circ \phi_t$. Each homeomorphism $\phi_t \colon Y \to Y$ extends to a self-homeomorphism of the entire sphere $S^2$ since all but at most one point of the set $S^2 - Y$ are isolated. Moreover, it is straightforward to check that $\phi_t|f^{-1}(A) \cup (S^2 - X) = \phi_0|f^{-1}(A) \cup (S^2 - X)$ for each $t \in \I$. In other words, the homotopy $(\phi_t)_{t \in \I}$ can be viewed as an isotopy in $S^2$ rel.\ $f^{-1}(A) \cup (S^2 - X)$.
    
    At the same time, $\varphi_0 \circ f = \varphi_0 \circ f \circ \phi_0$ and, therefore, we have the following:
    $$
        \varphi_1 \circ f = \varphi_0 \circ f \circ \phi_1 = (\varphi_0 \circ f \circ \phi_0) \circ \phi_0^{-1} \circ \phi_1 = \varphi_0 \circ f \circ \phi_0^{-1} \circ \phi_1 = \widetilde{f} \circ (\psi_0 \circ \phi_0^{-1} \circ \phi_1).
    $$
    Thus, we can set $\psi_1 := \psi_0 \circ \phi_0^{-1} \circ \phi_1$, and $(\psi_0 \circ \phi_0^{-1} \circ \phi_t)_{t \in \I}$ provides the required isotopy rel.\ $f^{-1}(A) \cup (S^2 - X)$. Clearly, if $p \in S^2$ is an accumulation point of the set $f^{-1}(A)$, then $p \not \in X$, which implies $\overline{f^{-1}(A)} \subset f^{-1}(A) \cup (S^2 - X)$.
    Finally, $\psi_1$ is orientation-preserving since $f$ and $\widetilde{f}$ are local orientation-preserving homeomorphisms outside the sets of their critical~points.
\end{proof} 

\begin{corollary}\label{corr: homotopy lifting for curves}
    Let $f \colon X \to S^2$ be a topologically holomorphic map of finite type, where $X = S^2$ or $X = S^2 - \{e\}$, $e \in S^2$, and $A \subset S^2$ be a finite set  containing $S_f$. Suppose that $\gamma_0$ is a simple closed curve in $S^2 - A$, and let $\widetilde{\gamma}_0 \subset f^{-1}(\gamma)$ be a simple closed curve with $\deg(f|\widetilde{\gamma}_0\colon \widetilde{\gamma}_0 \to \gamma_0) = d$. If $\gamma_1$ is a simple closed curve that is homotopic in $S^2 - A$ to $\gamma_0$, then there exists a simple closed curve $\widetilde{\gamma}_1 \subset f^{-1}(\gamma_1)$ such that $\widetilde{\gamma}_0$ and $\widetilde{\gamma}_1$ are homotopic in $X - f^{-1}(A) \subset S^2 - \overline{f^{-1}(A)}$ and $\deg(f|\widetilde{\gamma}_1\colon \widetilde{\gamma}_1 \to \gamma_1) = d$.
\end{corollary}

\begin{proof}
    According to \cite[Theorem A.3]{BuserGeometry} (see also \cite[Sections 1.2.5 and 1.2.6]{farb_margalit}), there exists an isotopy $(\varphi_t)_{t \in \I}$ rel.\ $A$ in $S^2$ such that $\varphi_0 = \id_{X}$ and $\varphi_1(\gamma_0) = \gamma_1$. Since $\varphi_0 =~\id_{S^2}$ is orientation-preserving, then $\varphi_t$ is also orientation-preserving for each $t \in \I$. By Proposition~\ref{prop: isotopy lifting property}, there exists a homeomorphism $\psi_1 \in \Homeo^+(S^2, f^{-1}(A) \cup (S^2 - X))$ such that $\varphi_1 \circ f = f \circ \psi_1$. Thus, we can take $\widetilde{\gamma}_1 := \psi_1(\widetilde{\gamma}_0)$. Finally, $X - f^{-1}(A) \subset S^2 - \overline{f^{-1}(A)}$, since any accumulation point $p \in S^2$ of the set $f^{-1}(A)$ cannot be in $X$.
\end{proof}

Due to the Uniformization Theorem and item~\eqref{it: pulling back} of Definition~\ref{def: top hol}, in the case when $X = S^2$, a topologically holomorphic map $f \colon X \to S^2$ can be written as $f = \varphi \circ g \circ \psi^{-1}$, where $g \colon \widehat{\C} \to \widehat{\C}$ is a non-constant rational map and $\varphi, \psi \colon S^2 \to \widehat{\C}$ are orientation-preserving homeomorphisms. In fact, in this case $f \colon S^2 \to S^2$ is simply a \textit{branched self-covering} of $S^2$, which is always of finite type and has finite topological degree.

Similarly, in the case when $X = S^2 - \{e\}$, we can write $f$ as $\varphi \circ g \circ \psi^{-1}$ such that $g \colon R \to \widehat{\C}$ is a non-constant meromorphic map, $\varphi \colon S^2 \to \widehat{\C}$ and $\psi \colon S^2 - \{e\} \to R$ are orientation-preserving homeomorphisms, where $R = \C$ or $R = \D$. Suppose that the map $f$ is of finite type. Then the image of $\psi$ above does not depend on the choice of the homeomorphism~$\varphi$ (see \cite[pp.~3-4]{geometric_function_theory}; it essentially follows from Proposition \ref{prop: isotopy lifting property} and some well-known facts from the theory of \textit{quasiconformal mappings}). Thus, finite-type topologically holomorphic maps for which the image of $\psi$ is $\C$ are referred to as \textit{parabolic type} maps, while those for which the image of $\psi$ is $\D$ are called \textit{hyperbolic type} maps.

Further, we assume that every topologically holomorphic map $f\colon S^2 \dto S^2$ we consider either has no essential singularities or is a finite-type topologically holomorphic map of parabolic type. Definition~\ref{def: top hol} and Great Picard's Theorem imply that in any neighborhood of an essential singularity $e$ of such a map $f$, every value is attained infinitely often with at most two exceptions. In particular, $f$  can have at most two \textit{omitted values}, i.e., points~$p$ in $S^2$ such that the preimage $f^{-1}(p)$ is empty. Furthermore, each omitted value is an asymptotic value of~$f$. Additionally, observe that if $A \subset S^2$ is a finite set with $|A| \geq 3$ and $S_f \subset A$, then the~restriction
$$
    f|S^2 - \overline{f^{-1}(A)}\colon S^2 - \overline{f^{-1}(A)} \to S^2 - A
$$
is a covering map. Note that the closure $\overline{f^{-1}(A)}$ equals $f^{-1}(A)$ if $f \colon S^2 \to S^2$ has no essential singularity. Otherwise, $\overline{f^{-1}(A)}$ consists of $f^{-1}(A)$ and the essential singularity $e \in S^2$ of $f$ due to Great Picard's Theorem and the assumption $|A| \geq 3$.

We say that a topologically holomorphic map $f \colon S^2 \dto S^2$ is \textit{transcendental} if it has an essential singularity. Given our previous assumptions on the map $f$, this is equivalent to saying that $f$ has infinite topological degree. The map $f \colon S^2 \dto S^2$ is called \textit{entire} if either $f$ has finite topological degree and there exists a point $p \in S^2$ such that $f^{-1}(p) = \{p\}$ (in which case $f$ is called a \textit{topological polynomial}), or $f$ has infinite topological degree and $f^{-1}(e) = \emptyset$, where $e$ is the essential singularity of $f$. We can view entire topologically holomorphic maps as topologically holomorphic self-maps of~$\R^2$.

\subsection{Thurston maps} \label{subsec: thurston maps}
Let $f \colon S^2 \dto S^2$ be a topologically holomorphic map. Then the \textit{postsingular set} $P_f$ of the map $f$ is defined as
$$
    P_f := \{q \in S^2: q = f^{\circ n}(p) \text{ for some } n \geq 0 \text{ and } p \in S_f\}.
$$

In other words, the postsingular set $P_f$ is the union of all forward orbits of the singular values of $f$. It is worth noting that some of these orbits might terminate after several iterations if a singular value reaches the essential singularity of the map~$f$.

We say that $f \colon S^2 \dto S^2$ is \textit{postsingularly finite} (\textit{psf} in short) if the set $P_f$ is finite, i.e., $f$ has finitely many singular values and each of them eventually becomes periodic or lands on the essential singularity of~$f$ under the iteration. 
Postsingularly finite topologically holomorphic maps of finite degree are also called \textit{postcritically finite} (\textit{pcf} in short), and their \textit{postsingular values} are called \textit{postcritical}, as their singular values are always critical.

Now we are ready to state one of the key definitions of this section.

\begin{definition}\label{def: thurston map}
    A non-injective topologically holomorphic map $f\colon S^2 \dto S^2$ is called a \textit{Thurston map} if it is postsingularly finite and either $f$ has no essential singularity or $f$ is a parabolic type map.
    
    Given a finite set $A \subset S^2$ such that $P_f \subset A$ and every $a \in A$ is either the essential singularity of $f$ or $f(a) \in A$, we call the pair $(f, A)$ a \textit{marked Thurston map} and $A$ its~\textit{marked~set}.
\end{definition}

We often consider marked Thurston maps in the same way as usual Thurston maps and use the notation $f\colon (S^2, A) \righttoleftarrow$ while still assuming that $f$ might not be defined on the entire sphere $S^2$. If no specific marked set is mentioned, we assume it to be $P_f$. Note that when the marked set $A$ contains the essential singularity of the map $f$, the set $A$ is not forward invariant with respect to $f$ in the usual sense. However, if $|A| \geq 3$, then~$A \subset \overline{f^{-1}(A)}$.

The dynamics of a Thurston map on its marked set or some other finite subsets of $S^2$ can also be represented graphically, in a way that turns out to be useful in study. Suppose that $f \colon S^2 \dto S^2$ is a Thurston map and $A \subset S^2$ is a finite set such that every $a \in A$ is either the essential singularity of $f$ or $f(a) \in A$. Then the \textit{dynamical portrait} of the map $f$ on the set $A$ is a directed abstract graph with the vertex set $A$, where for each vertex $v \in A$ that is not the essential singularity of $f$, there is a unique directed edge from $v$ to $f(v)$, and if $v \in A$ is the essential singularity of $f$, there are no outgoing edges from $v$. If the set $A$ coincides with the postsingular set $P_f$, the dynamical portrait of $f$ on the set $A$ is called the \textit{postsingular portrait} of the Thurston map~$f$.

\begin{definition}\label{def: isotopy of thurston maps}
    Two Thurston maps $f_1\colon (S^2, A) \righttoleftarrow$ and $f_2 \colon (S^2, A) \righttoleftarrow$ are called \textit{isotopic} (\textit{rel.\ $A$}) if there exists $\phi \in \Homeo_0^+(S^2, A)$ such that $f_1 = f_2 \circ \phi$.
\end{definition}

\begin{remark}\label{rem: isotopy of thurston maps}
    Let $f_1\colon (S^2, A) \righttoleftarrow$ and $f_2 \colon (S^2, A) \righttoleftarrow$ be two Thurston maps satisfying the relation $\phi_1 \circ f_1 = f_2 \circ \phi_2$ for some $\phi_1, \phi_2 \in \Homeo_0^+(S^2, A)$. Then it follows from Proposition~\ref{prop: isotopy lifting property} that $f_1$ and $f_2$ are isotopic rel.\ $A$.
\end{remark}

The notion of isotopy for Thurston maps depends on their common marked set. Consequently, we sometimes refer to isotopy \textit{relative $A$} (or rel.\ $A$ for short) to specify which marked set is being considered. This applies to other notions introduced below that also depend on the choice of the marked set.

We say that two (marked) Thurston maps are \textit{combinatorially equivalent} if they are ``topologically conjugate up to isotopy'':

\begin{definition}\label{def: comb equiv}
    Two Thurston maps $f_1 \colon (S^2, A_1) \rto$ and $f_2\colon (S^2, A_2) \rto$ are called combinatorially (or \textit{Thurston}) equivalent if there exist two Thurston maps $\widetilde{f}_1 \colon (S^2, A_1) \rto$ and $\widetilde{f}_2\colon (S^2, A_2) \rto$ such that:
    \begin{itemize}
        \item $f_i$ and $\widetilde{f}_i$ are isotopic rel.\ $A_i$ for each $i = 1, 2$, and

        \item $\widetilde{f}_1$ and $\widetilde{f}_2$ are conjugate via a homeomorphsim $\phi \in \Homeo^+(S^2)$, i.e., $\phi \circ \widetilde{f}_1 = \widetilde{f}_2 \circ \phi$, such that $\phi(A_1) = A_2$.
    \end{itemize}
\end{definition}

\begin{remark}\label{rem: comb equiv}
    Definition \ref{def: comb equiv} can be reformulated in a more classical way. Thurston maps $f_1 \colon (S^2, A_1) \rto$ and $f_2\colon (S^2, A_2) \rto$ are combinatorially equivalent if and only if there exist two homeomorphisms $\phi_1, \phi_2 \in \Homeo^+(S^2)$ such that $\phi_1(A_1) = \phi_2(A_1) = A_2$, $\phi_1$ and $\phi_2$ are isotopic rel.\ $A$, and $\phi_1 \circ f_1 = f_2 \circ \phi_2$.
\end{remark}

\begin{remark}
    If $A_1 = P_{f_1}$ and $A_2 = P_{f_2}$, then the condition that $\phi(A_1) = A_2$ in Definition~\ref{def: comb equiv} and the condition $\phi_1(A_1) = \phi_2(A_1) = A_2$ in Remark \ref{rem: comb equiv} can be removed since they are automatically satisfied if all other conditions hold.
\end{remark}

A Thurston map $f\colon (S^2, A)\righttoleftarrow$ is said to be \textit{realized} if it is combinatorially equivalent to a postsingularly finite holomorphic map $g \colon (\widehat{\C}, P) \rto$. If $f \colon (S^2, A) \rto$ is not realized, we say that it is \textit{obstructed}.

Let $A \subset S^2$ be a finite set. We say that a simple closed curve $\gamma \subset S^2 - A$ is \textit{essential} in $S^2 - A$ if each connected component of $S^2 - \gamma$ contains at least two points of the set $A$. In other words, $\gamma$ is essential in $S^2 - A$ if it cannot be shrinked to a point via a~homotopy~in~$S^2 - A$.

\begin{definition}\label{def: levy cycle}
Let $f \colon (S^2, A) \rto$ be a Thurston map. We say that a simple closed curve~$\gamma$ \textit{forms a Levy cycle} for $f \colon (S^2, A) \rto$ if $\gamma$ is essential in $S^2 - A$ and there exists another simple closed curve $\widetilde{\gamma} \subset f^{-n}(\gamma)$ for some $n \geq 1$ such that $\gamma$ and $\widetilde{\gamma}$ are homotopic in $S^2 - A$ and $\deg(f^{\circ n}| \widetilde{\gamma} \colon \widetilde{\gamma} \to \gamma) = 1$.
\end{definition}

If $n = 1$ in Definition \ref{def: levy cycle}, then $\gamma$ is called a \textit{Levy fixed} (or \textit{Levy invariant}) curve. Levy fixed curve $\gamma$ is called \textit{weakly degenerate} if $f$ is injective on one of the connected components of~$S^2 - \widetilde{\gamma}$. If additionally the image of this connected component $U$ under $f$ contains the same points of the set $A$ as $U$, i.e., $U \cap A = f(U) \cap A$, we say that $\gamma$ is a \textit{degenerate} Levy fixed curve for the Thurston map $f \colon (S^2, A) \rto$.

The following observation is widely known in the context of finite degree Thurston maps \cite[Exercise 10.3.6]{Hubbard_Book_2}, and its proof extends to the case of transcendental Thurston maps as well (see Section \ref{subsec: levy cycles} for the proof).

\begin{proposition}\label{prop: levy cycles are obstructions}
    Let $f\colon (S^2, A) \rto$ be a Thurston map. If there exists a simple closed curve $\gamma \subset S^2$ forming a Levy cycle for $f \colon (S^2, A) \rto$, then $f$ is obstructed rel.\ $A$.
\end{proposition}

We require one more notion of equivalence between Thurston maps.

\begin{definition}\label{def: huwritz equiv}
    We say that two Thurston maps $f_1\colon (S^2, A)\rto$ and $f_2 \colon (S^2, A) \rto$ are \textit{pure Hurwitz equivalent} (or simply \textit{Hurwitz equivalent}) if there exist two homeomorphisms $\phi_1, \phi_2 \in \Homeo^+(S^2, A)$ such that $\phi_1 \circ f_1 = f_2 \circ \phi_2$. 
\end{definition}

If $f \colon (S^2, A) \rto$ is a Thurston map, then the \textit{Hurwitz class} $\Hurw_{f, A}$ of $f$ is the union of all Thurston maps with the marked set $A$ that are Hurwitz equivalent to $f$. If $A$ coincides with the postsingular set $P_f$ of $f$, we simply use the notation~$\Hurw_f$. We say that a Thurston map $f \colon (S^2, A) \rto$ is \textit{totally unobstructed}~if~every~Thurston~map~in~$\Hurw_{f, A}$~is~unobstructed.

\begin{remark}\label{rem: at most three postsingular values}
    According to \cite[Proposition 2.3]{farb_margalit}, if two orientation-preserving homomeomorphisms $\varphi \colon S^2 \to \widehat{\C}$ and $\psi \colon S^2 \to \widehat{\C}$ agree on the set $A \subset S^2$ with $|A| \leq 3$, they are isotopic rel.\ $A$. This observation can be used to show that any Thurston map $f\colon (S^2, A) \rto$ is realized when the marked set $A$ contains three or fewer points. Similarly, in thise case, one can show that $\Hurw_{f, A}$ consists of a single Thurston map up to a combinatorial equivalence rel.\ $A$. However, when $|A| = 4$, the question of realizability already becomes significantly more challenging. 
\end{remark}



\subsection{Teichm\"uller and moduli spaces} \label{subsec: teichmuller spaces}

Let $A \subset S^2$ be a finite set containing at least three points. Then the \textit{Teichm\"{u}ller space of the sphere $S^2$ with the marked set $A$} is defined as
$$
    \T_A := \{\varphi\colon S^2 \rightarrow \widehat{\C} \text{ is an orientation-preserving homeomorphism}\} / \sim
$$
where $\varphi_1 \sim \varphi_2$ if there exists a M\"obius transformation $M$ such that $\varphi_1$ is isotopic rel.\ $A$ to~$M \circ \varphi_2$. 

Similarly, we define the \textit{moduli space of the sphere $S^2$ with the marked set $A$}:
$$
    \M_A := \{\eta\colon A \rightarrow \widehat{\C} \text{ is injective}\} / \sim,
$$
where $\eta_1 \sim \eta_2$ if there exists a M\"obius transformation $M$ such that $\eta_1 = M \circ \eta_2$. 

Further, $[\cdot]$ denotes an equivalence class corresponding to a point of either the Teichm\"uller space $\T_A$ or the moduli space $\M_A$.
Note that there is an obvious map $\pi \colon \T_A \to \M_A$ defined as $\pi([\varphi]) = [\varphi|A]$. According to \cite[Proposition 2.3]{farb_margalit}, when $|A| = 3$, both the Teichm\"uller space $\T_A$ and the moduli space $\M_A$ are just single points. Therefore, for the rest of this section, we assume that $|A| \geq 4$. 

It is known that the Teichm\"uller space $\T_A$ admits a complete metric $d_T$, known as the \textit{Teichm\"uller metric} \cite[Proposition~6.4.4]{Hubbard_Book_1}. Moreover, with respect to the topology induced by this metric, $\T_A$ is a contractible space \cite[Corollary 6.7.2]{Hubbard_Book_1}.  At the same time, both $\T_A$ and $\M_A$ admit structures of $(|A| - 3)$-complex manifolds (see \cite[Theorem 6.5.1]{Hubbard_Book_1}) so that the map $\pi \colon \T_A \to \M_A$ becomes a holomorphic universal covering map \cite[Section 10.9]{Hubbard_Book_2}. 

Moreover, the complex structure of $\M_A$ is quite explicit in the general case. Let $A = \{a_1, a_2, \dots, a_k, a_{k + 1}, a_{k + 2}, a_{k + 3}\}$, $ k \geq 1$, where the indexing of the points of $A$ is chosen arbitrarily. Define the map $h \colon \M_A \to \C^{k} - \L_k$ by
$$
    h([\varphi]) = (\varphi(a_1), \varphi(a_2), \dots, \varphi(a_k)),
$$
where the representative $\varphi \colon S^2 \to \widehat{\C}$ is chosen so that $\varphi(a_{k + 1}) = 0, \varphi(a_{k + 2}) = 1$, and $\varphi(a_{k + 3}) = \infty$, and where $\L_k$ is the subset of $\C^k$ defined by
$$
    \L_k := \{(z_1, z_2, \dots, z_k) \in \C^k: z_i = z_j \text{ for some $i \neq j$}, \text{ or } z_i = 0, \text{ or } z_i = 1\}.
$$
It is known that the map $h$ provides a biholomorphism between $\M_A$ and $\C^k - \L_k$ (see \cite[Section 10.9]{Hubbard_Book_2}).

Our focus in this paper is on the case when $|A| = 4$. In this situation, the Teichm\"uller space~$\T_A$ is biholomorphic to~$\D$, with the metric $d_T$ coinciding with the usual hyperbolic metric on $\D$; see \cite[Corollary 6.10.3 and Theorem 6.10.6]{Hubbard_Book_1}. Furthermore, the moduli space~$\M_A$ is biholomorphic to the three-punctured Riemann sphere $\Sigma$.

\subsection{Pullback maps} \label{subsec: pullback maps}
In this section, we illustrate how the notions introduced in Section~\ref{subsec: teichmuller spaces} can be applied for studying the properties of Thurston maps. Most importantly, using Definition~\ref{def: top hol} and Proposition \ref{prop: isotopy lifting property}, we can introduce the following crucial concept (see \cite[Proposition 2.21]{our_approx} for the proof; note that this is where the parabolic type condition in the Definition \ref{def: thurston map} plays a crucial role).

\begin{proposition}\label{prop: def of sigma map}
    Suppose that $f \colon (S^2, A) \righttoleftarrow$ is a Thurston map, or $f\in \Homeo^+(S^2)$ and $f(A) = A$, where $3 \leq |A| < \infty$. Let $\varphi \colon S^2 \to \widehat{\C}$ be an orientation-preserving homeomorphism. Then there exists an orientation-preserving homeomorphism $\psi\colon S^2 \to \widehat{\C}$ such that $g_\varphi := \varphi \circ f \circ \psi^{-1}\colon \widehat{\C} \dto \widehat{\C}$ is holomorphic. In other words, the following diagram commutes
    $$
    \begin{tikzcd}
        (S^2, A) \arrow[r,"\psi"] \arrow[d,"f", dashed] & (\widehat{\C}, \psi(A)) \arrow[d,"g_\varphi", dashed]\\
     (S^2, A) \arrow[r,"\varphi"] & (\widehat{\C}, \varphi(A))
    \end{tikzcd}
    $$
    The homeomorphism $\psi$ is unique up to post-composition with a M\"obius transformation. Different choices of $\varphi$ that represent the same point in $\T_A$ yield maps $\psi$ that represent the same point in $\T_A$. 
    
    In other words, we have a well-defined map $\sigma_f \colon \T_A \to \T_A$ such that $\sigma_f([\varphi]) = [\psi]$, called the pullback map (or the $\sigma$-map) associated with the Thurston map $f \colon (S^2, A) \rto$. As $\varphi$ ranges across all maps representing a single point in $\T_A$, the map $g_\varphi$ is uniquely defined up to pre- and post-composition with M\"obius transformations.
\end{proposition}

\begin{remark}\label{rem: functoriality}
    Let $\phi \colon S^2 \to S^2$ be an orientation-preserving homeomorphism with $\phi(A) = A$ and $|A| \geq 3$. It is straightforward to verify that if $\tau = [\varphi] \in \T_A$, then $\sigma_{\phi}([\varphi]) = [\varphi \circ \phi]$. Moreover, if $f \colon (S^2, A) \rto$ is a Thurston map, it is easy to see that $\sigma_{\phi \circ f} = \sigma_{f} \circ \sigma_{\phi}$ and $\sigma_{f \circ \phi} = \sigma_{\phi} \circ \sigma_{f}$.
\end{remark}

\begin{proposition}\label{prop: dependence}
    Suppose that we are in the setting of Proposition \ref{prop: def of sigma map}. If there exists a subset $B \subset A$, such that $S_f \subset B$ and $|B| = 3$, then $g_\varphi$, up to pre-composition with a M\"obius transformation, depends only on $\varphi|B$. Furthermore, if there exists a subset $C \subset A$ such that $|C| = 3$ and $C \subset \overline{f^{-1}(B)}$, then $g_\varphi$ is uniquely determined by $\varphi|B$ and $\psi|C$, and if $A \subset \overline{f^{-1}(B)}$, then $\sigma_f$ is a constant map.
\end{proposition}

\begin{proof}
    Suppose that $\varphi_1, \varphi_2, \psi_1, \psi_2 \colon S^2 \to \widehat{\C}$ are orientation-preserving homeomorphisms such that $\varphi_1|B = \varphi_2|B$, and the maps $g_{\varphi_1} = \varphi_1 \circ f \circ \psi_1^{-1}$ and $g_{\varphi_2} = \varphi_2 \circ f \circ \psi_2^{-1}$ are holomorphic possibly outside of single points in $\widehat{\C}$. One can easily see that we have the following:
    $$
        g_{\varphi_1} = (\varphi_1 \circ \varphi_2^{-1}) \circ g_{\varphi_2} \circ (\psi_2 \circ \psi_1^{-1}),
    $$
    where the homeomorphism $\varphi := \varphi_1 \circ \varphi_2^{-1}$ fixes each point of the set $\varphi_1(B) = \varphi_2(B)$. Since $|B| = 3$, \cite[Proposition 2.3]{farb_margalit} implies that $\varphi$ is isotopic rel.\ $\varphi_1(B)$ to $\id_{\widehat{\C}}$. According to Proposition \ref{prop: isotopy lifting property}, this isotopy can be lifted, leading to the relation $g_{\varphi_1} = g_{\varphi_2} \circ \psi$, where $\psi \in \Homeo^+(S^2)$. It is easy to see that the homeomorphism $\psi$ is~a~M\"obius~transformation.

    Now, suppose that there exists a subset $C \subset A$ such that $|C| = 3$ and $C \subset \overline{f^{-1}(B)}$, and $\psi_1|C = \psi_2|C$. Note that the homeomorphism $\psi$ is isotopic rel.\ $\overline{g_{\varphi_1}^{-1}(\varphi_1(B))} = \psi_1(\overline{f^{-1}(B)})$ to $\psi_2 \circ \psi_1^{-1}$ due to Proposition~\ref{prop: isotopy lifting property}. Consequently, $\psi|\psi_1(C) = (\psi_2 \circ \psi_1^{-1})|\psi_1(C) = \id_{\psi_1(C)}$. Since $\psi$ is a M\"obius transformation fixing three distinct points in $\widehat{\C}$, it must be the identity~$\id_{\widehat{\C}}$. Thus, the maps $g_{\varphi_1}$ and $g_{\varphi_2}$ coincide. 
    
    If $A \subset \overline{f^{-1}(B)}$, then $\psi_2 \circ \psi_1^{-1}$ is isotopic rel.\ $\psi_1(A)$ to the M\"obius transformation $\psi =~\id_{\widehat{\C}}$. Thus, $\sigma_f([\varphi_1]) = [\psi_1] = [\psi_2] = \sigma_f([\varphi_2])$ in the Teichm\"uller space $\T_A$, and the rest follows.
\end{proof}

The following observation provides the most crucial property of pullback maps (see \cite[Proposition 2.24]{our_approx} for the proof).

\begin{proposition}\label{prop: fixed point of sigma}
    A Thurston map $f\colon (S^2, A) \righttoleftarrow$ with $|A| \geq 3$ is realized if and only if the pullback map~$\sigma_f$ has a fixed point in the Teichm\"uller space $\T_A$.
\end{proposition}

To illustrate the principle formulated in Proposition \ref{prop: fixed point of sigma}, we present the following remark.


\begin{remark}\label{rem: constant sigma map}
    Let $f \colon (S^2, A) \rto$ be a Thurston map with $|A| \geq 3$, and suppose that there is a subset $B \subset A$ such that $S_f \subset B$, $|B| = 3$, and $A \subset\overline{f^{-1}(B)}$. Then $f$ is realized rel.\ $A$ because $\sigma_f$ is a constant map according to Proposition \ref{prop: dependence}. Additionally, by applying Remark \ref{rem: functoriality} and Proposition \ref{prop: dependence}, it is easy to show that the Hurwitz class $\Hurw_{f, A}$ consists of a single Thurston map, up to combinatorial equivalence rel.\ $A$.

    However, Thurston maps that satisfy these conditions are somewhat artificial. For such a marked Thurston map $f \colon (S^2, A) \rto$, it must hold $P_f \subset B$ and for every $a \in A - B$, either $a$ is the essential singularity of $f$ or $f(a) \in B$. For instance, this scenario is impossible for unmarked Thurston maps with at least four postsingular values.
\end{remark}

\begin{proposition}\label{prop: propeties of pullback map}
    Let $f \colon (S^2, A) \rto$ be a Thurston map with $|A| \geq 3$. Then the pullback map $\sigma_f$ is holomorphic.
\end{proposition}
\begin{proof}
    This result is rather well-known in the context of finite degree Thurston maps (see \cite[Section 1.3]{Buff} and \cite[Sections 10.6 and 10.7]{Hubbard_Book_2}), and it can be extended analogously to the transcendental setting (see, for instance, \cite[Lemma 3.3]{Astorg}).
\end{proof}

\begin{remark}\label{rem: further properties}
    Proposition \ref{prop: propeties of pullback map} and \cite[Corollary 6.10.7]{Hubbard_Book_1} imply that the map $\sigma_f$ is \textit{$1$-Lipschitz}, meaning  $d_T(\sigma_f(\tau_1), \sigma_f(\tau_2)) \leq d_T(\tau_1, \tau_2)$ for every $\tau_1, \tau_2 \in \T_A$.
    In fact, in many cases, such as when the Thurston map $f \colon (S^2, A) \rto$ is transcendental, it can be shown (see \cite[Section 3.2]{HSS}, \cite[Chapter 5.1]{Pfrang_thesis}, or \cite[Sections 2.3 and 3.1]{Astorg}) that $\sigma_f$ is actually \textit{distance-decreasing}, i.e., $d_T(\sigma_f(\tau_1), \sigma_f(\tau_2)) < d_T(\tau_1, \tau_2)$ for every distinct $\tau_1, \tau_2 \in \T_A$. This property of pullback maps can be used to obtain certain rigidity results for transcendental postsingularly finite meromorphic maps (cf. \cite[Corollary 10.7.8]{Hubbard_Book_2} and \cite[Proposition 2.26]{our_approx}).

    However, we do not require these results and the observation of Proposition \ref{prop: propeties of pullback map} for our further arguments since we mostly work with families of Thurston maps satisfying additional assumptions. For these families, we will directly observe all the properties mentioned above.
\end{remark}

\section{Hyperbolic tools} \label{sec: hyperbolic tools}

Let $U$ be a hyperbolic Riemann surface, and let $d_U$ denote the distance function of the hyperbolic metric on $U$. For any rectifiable curve $\alpha$ in $U$, we denote the length of~$\alpha$ with respect to the hyperbolic metric by~$\ell_{U}(\alpha)$. When we refer to $\gamma$ as a geodesic in $U$, we always mean that $\gamma$ is a geodesic with respect to the hyperbolic metric. Also, let $B_U(z, r)$ be the hyperbolic ball in~$U$ with center $z \in U$ and radius $r$. If $U$ is a subset of $\C$, then the hyperbolic metric on $U$, as a conformal metric, can be written as $\rho_U(z)|dz|$, where $\rho_U\colon U \to [0, +\infty)$ is the \textit{density} of the hyperbolic metric on~$U$.

For a holomorphic map $g \colon U \to V$ between two hyperbolic Riemann surfaces, we denote by $\|\mathrm{D}g(z)\|_U^V$ the norm of the derivative of $g$ with respect to the hyperbolic metrics on the domain~$U$ and the range $V$. More precisely, this norm is given by $\|\mathrm{D}g(z)\|_U^V = \|\mathrm{D}_zg(v)\|_V / \|v\|_U$,
where $v \in T_zU$ is any non-zero vector, and $\|\cdot\|_U$ represents the length of a tangent vector to~$U$ with respect to the hyperbolic metric. If $U = V$, we simply use the notation $\|\mathrm{D}g(z)\|_U$. 

Schwarz-Pick's lemma \cite[Proposition~3.3.4]{Hubbard_Book_1} implies that for a holomorphic map $g \colon U \to V$ between two hyperbolic Riemann surfaces, we have $\|\mathrm{D}g(z)\|_U^V \leq 1$ for every $z \in U$. Furthermore, if $g$ is a covering map, this inequality becomes an equality; otherwise, $g$ is \textit{locally uniformly contracting}, i.e., for every compact set $K \subset U$, there exists a constant $\lambda_K < 1$ such that $\|\mathrm{D}g(z)\|_U^V \leq \lambda_K$ for all $z \in K$. Suppose that $\alpha$ is a $C^1$-curve in~$U$ and $\|\mathrm{D}g(z)\|_U^V \leq \lambda$ for all $z \in \alpha$. Then it is straightforward to check that $\ell_{V}(g(\alpha)) \leq \lambda \ell_U(\alpha)$. In particular, Schwarz-Pick's lemma implies that the map $g\colon U \to V$ is always 1-Lipschitz and, if $g$ is not a covering map, then $g$ is \textit{locally uniformly distance-decreasing}.

We will be particularly interested in the case when the map $g$ mentioned above is simply the inclusion map $I \colon U \hookrightarrow V$. We are going to denote $\|\mathrm{D}I(z)\|_U^V$ by $c_U^V(z)$. Clearly, if $U \subset V \subset \C$, then $c_U^V(z) = \rho_V(z)/\rho_U(z)$.


\subsection{Levy cycles} \label{subsec: levy cycles}

Let $\delta$ be an essential simple closed curve in the punctured Riemann sphere $X := \widehat{\C} - P$, where $3 \leq |P| < \infty$. According to \cite[Proposition 3.3.8]{Hubbard_Book_1}, there exists a unique closed geodesic in $X$ that is homotopic (in $X$) to $\delta$. Note that this geodesic should be necessarily simple \cite[Proposition 3.3.9]{Hubbard_Book_1}. 

In order to illustrate the utility of the hyperbolic tools, we prove that Levy cycles are obstructions for Thurston maps of both finite and infinite degree.

\begin{proof}[Proof of Proposition \ref{prop: levy cycles are obstructions}]
    Let $\gamma$ be a simple closed curve forming a Levy cycle for the Thurston map~$f \colon (S^2, A) \rto$. We assume that $n = 1$ in Definition \ref{def: levy cycle}, i.e., $\gamma$ is a Levy fixed curve. The general case can be handled in a similar manner. 

    Suppose that $f\colon (S^2, A) \rto$ is combinatorially equivalent to a postsingularly finite holomorphic map $g \colon (\widehat{\C}, P) \rto$. Based on Definition \ref{def: comb equiv} and Remark \ref{rem: comb equiv}, it is easy to see that $g \colon (\widehat{\C}, P) \rto$ also has a Levy fixed curve $\delta \subset \widehat{\C} - P$. Note that $|P| \geq 4$ since, otherwise, $\delta$ would not be essential in $\widehat{\C} - P$. According to the previous discussion and Corollary~\ref{corr: homotopy lifting for curves}, we can assume that $\delta$ is a simple closed geodesic in $\widehat{\C} - P$. 
    
    Let $\widetilde{\delta}$ be a connected component of $g^{-1}(\delta)$ such that $\delta$ and $\widetilde{\delta}$ are homotopic in $\widehat{\C} - P$, and $\deg(g|\widetilde{\delta}\colon \widetilde{\delta} \to \delta) = 1$. By Schwarz-Pick's lemma, it follows that
    $$
        \ell_{\widehat{\C} - P}(\delta) = \ell_{\widehat{\C} - \overline{g^{-1}(P)}}(\widetilde{\delta}) > \ell_{\widehat{\C} - P}(\widetilde{\delta}),
    $$
    where the last inequality is strict since $\overline{g^{-1}(P)} - P$ is non-empty because of Great Picard's theorem and the Riemann-Hurwitz formula \cite[Appendix A.3]{Hubbard_Book_1}.
    However, $\delta$ is the unique geodesic in its homotopy class in $\widehat{\C} - P$. Thus, $\ell_{\widehat{\C} - P}(\delta) \leq \ell_{\widehat{\C} - P}(\widetilde{\delta})$, and it leads to a contradiction. 
\end{proof}

Let $\gamma$ be an essential simple closed curve in $S^2 - A$, and $\tau = [\varphi]$ be a point in the Teichm\"uller space $\T_A$. We define $l_{\gamma}(\tau)$ as the length of the unique hyperbolic geodesic in $\widehat{\C} - \varphi(A)$ that is homotopic in $\widehat{\C} - \varphi(A)$ to $\varphi(\alpha)$. Additionally, we introduce $w_{\gamma}(\tau) := \log l_{\gamma}(\tau)$. It is known that $w_{\gamma} \colon \T_A \to  \R$ is a $1$-Lipschitz function \cite[Theorem~7.6.4]{Hubbard_Book_1}.

Let $X$ be a hyperbolic Riemann surface, and suppose $\alpha \subset X$ is a simple closed geodesic with $\ell_{X}(\alpha) < \ell^*$, where $\ell^* := \log(3 + 2 \sqrt{2})$. In this case, we say that $\alpha$ is \textit{short}. As stated in \cite[Proposition 3.3.8 and Corollary~3.8.7]{Hubbard_Book_1}, two short simple closed geodesics on a hyperbolic Riemann surface $X$ are either disjoint and non-homotopic in $X$, or they coincide. Therefore, \cite[Proposition 3.3.8]{Hubbard_Book_1} implies that a punctured Riemann sphere $\widehat{\C} - P$, where $3 \leq |P| < \infty$, can have at most $|P| - 3$ distinct short simple closed geodesics.

The following result allows us to identify a Levy fixed curve for a Thurston map based on the behavior of the corresponding pullback map.

\begin{proposition}\label{prop: finding levy cycles}
    Let $f \colon (S^2, A) \rto$ be a Thurston map with $|A| = 4$, and $\tau = [\varphi]$ and $\sigma_f(\tau) = [\psi]$ be points in the Teichm\"uller space $\T_A$, where the representatives $\varphi, \psi \colon S^2 \to \widehat{\C}$ are chosen so that the map $g := \varphi \circ f \circ \psi^{-1} \colon \widehat{\C} \dto \widehat{\C}$ is holomorphic. Suppose that there exists an annulus $U \subset \widehat{\C}$ such that:
    \begin{itemize}
        \item each connected component of $\widehat{\C} - U$ contains two points of $\psi(A)$;
        \item $\mod(U) > 5\pi e^{d_0}/\ell^*$, where $d_0 = d_T(\tau, \sigma_f(\tau))$;
        \item $g$ is defined and injective on $U$.
    \end{itemize}
    Then $f \colon (S^2, A) \rto$ has a Levy fixed curve. Moreover, if $g$ is defined and injective outside a single connected component of $\widehat{\C} - U$, then $f \colon (S^2, A) \rto$ has a weakly degenerate Levy fixed~curve.
\end{proposition}

\begin{proof}
    Since $g$ is defined and injective on $U$, the annulus $U$ contains at most 4 points of the set $\overline{g^{-1}(\varphi(A))}$. Therefore, we can find a parallel subannulus $V$ of $U$ such that $\mod(V) \geq \mod(U) / 5 > \pi e^{d_0} \ell^*$, and $V$ does not contain any points of $\overline{g^{-1}(\varphi(A))}$. Denote $X := \widehat{\C} - \psi(A)$, $Y :=\widehat{\C} - \varphi(A)$, and $Z := \widehat{\C} - \overline{g^{-1}(\varphi(A))}$. In particular, $g|Z\colon Z \to X$ is a holomorphic covering map.

    Let $\alpha$ be a unique hyperbolic geodesic of $V$. It is known that $\alpha$ is a simple closed curve that forms a core curve of the annulus $V$, and its length in $V$ is given by $\ell_V(\alpha) = \pi / \mod(V)$; see \cite[Proposition 3.3.7]{Hubbard_Book_1}. Let $\beta$ denote the curve $g(\alpha)$. Since $\alpha$ is a simple closed curve and $g|V$ is injective, then $\beta$ is also a simple closed curve. At the same time, $\beta$ must be essential in $Y$ since, otherwise, $\alpha$ would not be essential in $X$ (see, for instance, \cite[Theorems 5.10 and 5.11]{Forster}). Define $\widetilde{\alpha} := \psi^{-1}(\alpha)$ and $\widetilde{\beta} := \varphi^{-1}(\beta)$, both of which are essential simple closed curves in $S^2 - A$. It is straightforward to verify that $f(\widetilde{\alpha}) = \widetilde{\beta}$ and $\deg(f|\widetilde{\alpha}\colon \widetilde{\alpha} \to \widetilde{\beta}) = 1$.

    According to Schwarz-Pick's lemma and the choice of $\alpha \subset Z$, we have the following inequality:
    $$
        \ell_{X}(\alpha) \leq \ell_{Z}(\alpha) \leq \ell_V(\alpha) = \pi / \mod(V) < e^{-d_0}\ell^*.
    $$
    Therefore, since the function $w_{\widetilde{\alpha}} \colon \T_A \to \R$ is 1-Lipschitz, it follows that $l_{\widetilde{\alpha}}(\tau) < \ell^*$. Hence, there exists a simple closed geodesic $\delta$ in $Y$, homotopic in $Y$ to $\varphi(\widetilde{\alpha})$, such that $\ell_{Y}(\delta) <~\ell^*$. At the same time, $\ell_{Y}(\beta) = \ell_{Z}(\alpha) < e^{-d_0}\ell^*$. Therefore, since both $\beta$ and $\delta$ are essential in $Y$, they are homotopic in $Y$ to short simple closed geodesics.
    However, as discussed previously, the four-punctured Riemann sphere $Y$ can have only one short simple closed geodesic. This implies that $\beta$ and $\delta$ are homotopic in $Y$, which in turn means that the curves $\widetilde{\alpha}$ and $\widetilde{\beta}$ are homotopic in $S^2 - A$. Hence, $\widetilde{\beta}$ provides a Levy fixed curve for the Thurston map $f \colon (S^2, A) \rto$. It is also straightforward to verify that if $g$ is defined and injective outside a single connected component of $\widehat{\C} - U$, then this Levy fixed curve is weakly degenerate.
\end{proof}

The following result guarantees the uniqueness of a Levy fixed curve for a Thurston map with four marked points, given a specific technical condition.

\begin{proposition}\label{prop: uniqueness of levy cycles}
    Let $f\colon (S^2, A) \rto$ be a Thurston map with $|A| = 4$ that has a Levy fixed curve. Suppose that there exists a point $\tau$ in the Teichm\"uller space $\T_A$ such that the sequence $(\pi(\sigma_f^{\circ n}(\tau)))$ eventually leaves every compact subset of the moduli space $\M_A$. Then $f \colon (S^2, A) \rto$ has a unique Levy fixed curve up to homotopy in~$S^2 - A$.
\end{proposition}

\begin{proof}
    Let $\gamma$ be a Levy fixed curve for the Thurston map $f \colon (S^2, A) \rto$. Choose a point $\mu \in \T_A$ such that $l_{\gamma}(\mu) < \ell^*$. Define $\tau_n := \sigma_f^{\circ n}(\tau) = [\varphi_n]$ and $\mu_n := \sigma_f^{\circ n}(\mu) = [\psi_n]$ for all $n \geq 0$. According to Proposition \ref{prop: def of sigma map}, we can assume that $g_n := \psi_n \circ f \circ \psi_{n + 1}^{-1} \colon \widehat{\C} \dto \widehat{\C}$ is holomorphic for every $n \geq 0$. We also define $X_n := \widehat{\C} - \varphi_n(A)$, $Y_n := \widehat{\C} - \psi_n(A)$, and $Z_n := \widehat{\C} - \overline{g_n^{-1}(\psi_{n}(A))} = \widehat{\C} - \overline{\psi_{n + 1}(f^{-1}(A))}$ for all $n \geq 0$. In particular, $g_n|Z_n \colon Z_n \to Y_n$ is a holomorphic covering map.

    \begin{claim}
        We have $l_{\gamma}(\mu_n) < \ell^*$ for all $n \geq 0$.
    \end{claim}

    \begin{subproof}
        Clearly, $l_{\gamma}(\mu_0) < \ell^*$ by the choice of $\mu$. We will prove that $l_{\gamma}(\mu_1) < \ell^*$, and the rest easily follows by induction on $n$. There exists a short simple closed geodesic $\beta$ in $Y_0$ such that $\beta$ and $\psi_0(\gamma)$ are homotopic in $Y_0$, and $\ell_{Y_0}(\beta) = l_{\gamma}(\mu_0)$. Define $\widetilde{\beta} := \psi^{-1}_0(\beta)$. Since $\widetilde{\beta}$ and $\gamma$ are homotopic in $S^2 - A$, then by Corollary \ref{corr: homotopy lifting for curves}, there exists a simple closed curve $\widetilde{\alpha} \subset f^{-1}(\widetilde{\alpha})$ that is homotopic to $\gamma$ in $S^2 - A$, since $\gamma$ is a Levy fixed curve for $f \colon (S^2, A) \rto$, and $\deg(f|\widetilde{\alpha} \colon \widetilde{\alpha} \to \widetilde{\beta}) = 1$. Now, define $\alpha := \psi_1(\widetilde{\alpha})$, which is homotopic in $Y_1$ to $\psi_1(\gamma)$. Clearly, $g(\alpha) = \beta$, $\deg(g|\alpha \colon \alpha \to \beta) = 1$, and $\alpha \subset Z_0$. Therefore, by Schwarz-Pick's lemma, we have
        $$
            \ell_{Y_1}(\alpha) \leq \ell_{Z_0}(\alpha) = \ell_{Y_0}(\beta) < \ell^*.
        $$
        Finally, $l_{\gamma}(\mu_1) = l_{\widetilde{\alpha}}(\mu_1) < \ell^*$.
    \end{subproof}
    

    Since $(\pi(\tau_n))$ eventually leaves every compact subset of $\M_A$, Mumford's compactness theorem \cite[Theorem~7.3.3]{Hubbard_Book_1} states that the length of the shortest simple closed geodesic~$\delta_n$ in $X_n$ tends to zero. Then $w_{\varphi_{n}^{-1}(\delta_{n})}(\mu_{n}) < \log \ell^*$ for any sufficiently large $n$, given that $w_{\varphi_n^{-1}(\delta)} \colon \T_A \to \R$ is 1-Lipschitz and $d_T(\mu_n, \tau_n) \leq d_T(\tau, \mu)$ by Proposition \ref{prop: propeties of pullback map} and Schwarz-Pick's lemma. Thus, $\psi_n(\gamma)$ and $\psi_n(\varphi_n^{-1}(\delta_n))$ are homotopic in $Y_n$ to short simple closed geodesics. Since $Y_n =  \widehat{\C} - \psi_n(A)$ is a four-punctured Riemann sphere, it follows that $\gamma$ and $\varphi_{n}^{-1}(\delta_{n})$ are homotopic in $S^2 - A$ for any $n$ large enough.

    This argument applies to any Levy fixed curve of $f \colon (S^2, A) \rto$. Therefore, it has to be unique up to homotopy in $S^2 - A$.
\end{proof}

\begin{remark}\label{rem: th obstruction}
 If an obstructed Thurston map $f\colon (S^2, A) \rto$ with $|A| \geq 3$ has a finite degree and \textit{hyperbolic orbifold}, then it is known (see \cite[Section~10.9 and Lemma 10.11.9]{Hubbard_Book_2} or \cite[Proof of Theorem~2.3,~p.~20]{Nikita}) that every point $\tau$ in the Teichm\"uller space $\T_A$~satisfies the condition of Proposition \ref{prop: uniqueness of levy cycles}.
    
\end{remark}

\begin{remark}
    Suppose that we are in the setting of Proposition \ref{prop: uniqueness of levy cycles}. It is clear that it suffices to require that the sequence $(\pi(\sigma_f^{\circ n}(\tau)))$ admits a subsequence eventually leaving every compact subset of the moduli space $\M_A$.
\end{remark}

\subsection{Estimating contraction} \label{subsec: estimating contraction}

The following proposition provides an estimate for the contraction of an inclusion between hyperbolic Riemann surfaces. Although this result is well-known \cite[Theorem 2.25]{McM_renorm}, we include the proof here for the sake of completeness.

\begin{proposition}\label{prop: mcmullens bound}
    Suppose that $U$ and $V$ are two hyperbolic Riemann surfaces such that $U \subsetneq V$. Let $z \in U$ and $s := d_V(z, V - U)$. Then
    $$
        c_U^V(z) \leq \frac{2r |\log r|}{1 - r^2},
    $$
    where $r := \tanh(s/2)$.
\end{proposition}

\begin{proof}
    Suppose that $z_0 \in V - U$ is a point such that $s = d_V(z, z_0)$. Let $p \colon \D \to V$ be a holomorphic universal covering with $p(0) = z_0$. Let us chose $w \in p^{-1}(z)$ such that $d_{\D}(0, w) =~s$. Finally, we denote by $\widetilde{U} \subset \D$ the connected component of $w$ in $p^{-1}(U)$. Then we have the following commutative diagram:
   \begin{center}
       \begin{tikzcd}
           \D^* \arrow[dr, hook] & \widetilde{U} \arrow[r, "p|\widetilde{U}"] \arrow[d, hook] \arrow[l, hook] & U \arrow[d, hook]\\
            & \D \arrow[r, "p"] & V
       \end{tikzcd}
   \end{center}
    Applying Schwarz-Pick's lemma and recalling that $p \colon \D \to V$ and $p|\widetilde{U} \colon \widetilde{U} \to U$ are holomorphic covering maps, we obtain the following:
    $$
        c_U^V(z) = c_{\widetilde{U}}^{\D}(z) = \frac{\rho_{\D}(w)}{\rho_{\widetilde{U}}(w)} \leq \frac{\rho_{\D}(w)}{\rho_{\D^*}(w)}.
    $$
    The rest easily follows since $\rho_{\D}(w) = 2/(1 - |w|^2)$, $\rho_{\D^*}(w) = 1/|w\log|w||$ (\cite[Example 3.3.2]{Hubbard_Book_1}), and $s = d_{\D}(0, w) = 2 \tanh^{-1}(|w|)$ due to \cite[Exercise~2.1.8]{Hubbard_Book_1}.
\end{proof}

\begin{remark}\label{rem: mcmullens bound}
    Suppose that we are in the setting of Proposition \ref{prop: mcmullens bound}. 
    It shows that there exists an upper bound $\lambda(s) < 1$ for $c_U^V(z)$ depending only on $s = d_V(z, V - U)$ and not on $U$, $V$, or $z$. Moreover, $\lambda(s) \sim |s \log(s)|$ as $s \to 0$. In particular, $\lambda(s) \to 0$ as~$s \to 0$.
\end{remark}

Proposition \ref{prop: mcmullens bound} and Remark \ref{rem: mcmullens bound} allow us to estimate the contraction of certain inclusions between countably and finitely punctured Riemann spheres (see \cite[Lemma 2.1]{RempeActa} for the result of a similar nature).

\begin{proposition}\label{prop: inclusion of punctured spheres}
    Suppose that $(z_n)$ is a sequence in $\C$ such that $\lim_{n \to \infty} z_n = \infty$ and $\lim_{n \to \infty} \log|z_{n + 1}|/\log|z_n| = 1$. Let $P$ be a finite set consisting of at least two elements of~$(z_n)$.
    Then $c_U^V(z)$ converges to 0 as $|z|$ tends to $\infty$ for $z \in U$, where $U := \C - \{z_n: n \in \N\}$ and~$V := \C - P$.
\end{proposition}

\begin{proof}
    Without loss of generality, we can assume that $P \subset \D$ and the sequence of absolute values $(|z_n|)$ is non-decreasing. Let $Q := \{z_n: n \in \N\}$, and let $p \colon \H \to \C - \overline{\D}$ be the holomorphic universal covering defined as $p(z) = \exp(-iz)$ for $z \in \H$. 
    
    According to Proposition \ref{prop: mcmullens bound} and Remark \ref{rem: mcmullens bound}, it suffices to demonstrate that for any $r > 0$, the set $\bigcup_{x \in Q - P} B_V(x, r)$ covers some punctured neighborhood of $\infty$ in $\widehat{\C}$. Alternatively, it is sufficient to show that for any $r > 0$, there exists $t > 0$ so that
    \begin{align}\label{align: inclusion}
        \bigcup_{x \in p^{-1}(Q - \overline{\D})}B_{\H}(x, r) \supset \H_{t} := \{z \in \C: \Im(z) > t\}.
    \end{align}
    Indeed, $p(\H_t) = \C - \overline{\D}_{e^{t}}$ and $p(B_{\H}(x, r)) \subset B_{V}(p(x), r)$ due to Schwarz-Pick's lemma.

    Note that for sufficiently large $n$, we have that $|z_n| > 1$ and 
    \begin{align}\label{align: preimage}
        \pi^{-1}(z_n) = \{2\pi k - \arg(z_n) + i\log(|z_n|): k \in \Z\} =: \{z_{n, k}: k \in \Z\}.
    \end{align}

    Observation (\ref{align: preimage}) implies that
    \begin{align}\label{align: diam}
        d_{n, k} := \mathrm{diam}_{\H}\left(\left\{z: 2\pi k \leq \Re(z) \leq 2\pi (k + 1), \Im(z_{n, k}) \leq \Im(z) \leq \Im(z_{n+1, k})\right\}\right) \leq\\
        \leq \log\left(\frac{\log|z_{n + 1}|}{\log|z_n|}\right) + \frac{2\pi}{\log|z_n|}.\nonumber
    \end{align}

    Estimate (\ref{align: diam}) is derived from the following well-known facts about the hyperbolic distance between two points $z, w \in \H$:
    \begin{itemize}
        \item if $\Re(z) = \Re(w)$, then $d_{\H}(z, w) = |\log(\Im(z)/\Im(w))|$, and
        \item if $\Im(z) = \Im(w)$, then $d_{\H}(z, w) \leq |\Re(z) - \Re(w)|/\Im(z)$. In fact, the upper estimate is the hyperbolic length of the horizontal segment connecting $z$ and $w$.
    \end{itemize}

     Finally, estimate (\ref{align: diam}) implies that $d_{n, k}$ tends to zero independently from $k$ when $n \to \infty$. Thus, inclusion (\ref{align: inclusion}) holds for certain $t > 0$, and the desired result follows.   
\end{proof}

\begin{remark}
    Suppose that we are in the setting of Proposition \ref{prop: inclusion of punctured spheres}. Assume that we know that $\log|z_{n+1}|/\log|z_n|$ is uniformly bounded from above but does not necessarily converge to $1$ as $n \to \infty$. Following the proof of Proposition \ref{prop: inclusion of punctured spheres}, one can show that there exists $\lambda < 1$ so that $c_U^V(z) < \lambda$ for any $z \in W \cap U$, where $W \subset \widehat{\C}$ is a neighborhood of $\infty$.
\end{remark}

As an application of Proposition \ref{prop: inclusion of punctured spheres}, we can obtain the following result.

\begin{proposition}\label{prop: coverings and inclusions}
    Suppose that $g \colon U \to V$ is a holomorphic covering map, where $U \subset V$ is a domain of $\widehat{\C}$ and $V = \widehat{\C} - P$ with $3 \leq |P| < \infty$. Let $x \in P$ be an accumulation point of the set $U$. Then exactly one of the following two possibilities is satisfied:
    \begin{enumerate}
        \item \label{it: removable singularity} $x$ is an isolated removable singularity of $g \colon U \to V$, or

        \item \label{it: contraction} $c_U^V(z)$ converges to $0$ as $|z - x|$ tends to 0 for $z \in U$.
    \end{enumerate}
\end{proposition}

\begin{proof}
    Without loss of generality, we can assume that $x = \infty$. Suppose that condition (\ref{it: removable singularity}) is not satisfied. We will prove that there exists a constant $a > 1$ such that for every sufficiently large $r > 0$, the annulus $A(r/a, ra) := \{z \in \C: r/a < |z| < ra\}$ contains a point of $\C - U$ (see \cite[Lemma 3.2]{LandingTheorem} for a similar result in the setting of entire maps). Suppose the contrary, i.e., that there exist sequences $(a_n)$ and $(r_n)$ such that $a_n > 1$, $a_n \to \infty$, $r_n \to \infty$, and $A_n := A(r_n/a_n, r_na_n) \subset U$ for every $n \in \N$.

    In particular, $g$ is defined and meromorphic on $A_n$, allowing us to consider the sequence of meromorphic functions $g_n \colon A(1/a_n, a_n) \to \widehat{\C}$, $n \in \N$, where $g_n(z) = g(r_nz)$ for every $z \in A(1/a_n, a_n)$. By Montel's theorem, the family $(g_n)$ is normal since each $g_n$ omits the values in the set $P$. Thus, there exists a subsequence of $(g_n)$ that converges locally uniformly on the punctured plane $\C^*$. Without loss of generality, we assume that this subsequence coincides with the original sequence $(g_n)$. Let $h \colon \C^* \to \widehat{\C}$ be the limiting function. If $h$ is not constant, then by Hurwitz's theorem \cite[p.~231]{Gamelin}, $h$ also omits the values in the set $P$. This leads to a contradiction because either $h$ would have an isolated essential singularity at $0$ or~$\infty$, contradicting Great Picard's theorem, or both $0$ and $\infty$ would be removable singularities or poles of $h$, and thus, $h$ is a rational map, which cannot omit any values.

    Therefore, $h$ is a constant map. Denote by $w \in \widehat{\C}$ its unique value. Choose a simply connected domain $W \ni w$ such that $\overline{W} \cap P = \{w\} \cap P$. From the previous discussion, $g(A(2r_n/a_n, a_n r_n/2)) \subset W$ for all $n$ large enough. Note that every connected component of $g^{-1}(W) \subset U$ is either a simply connected domain or a simply connected domain with a single puncture (see, for instance, \cite[Theorems 5.10 and 5.11]{Forster}). Hence, for every sufficiently large $n$, either all but at most one of the points in $\D_{a_n r_n / 2}$ belong to $g^{-1}(W)$, or the same is true for the set $\widehat{\C} - \overline{\D}_{2r_n/a_n}$. The first case cannot happen for infinitely many $n$, so there must be some $n \in \N$ for which the second case occurs. It immediately leads to a contradiction, as we initially assumed that $x$ is not an isolated removable singularity of the map $g \colon U \to V$.

    Now, we can take $z_n$ to be a point of $\widehat{\C} - U$ in $A(a^{2n}, a^{2n + 2})$ for every sufficiently large $n$. It is easy to see that the sequence $(z_n)$ satisfies the conditions of Proposition \ref{prop: inclusion of punctured spheres}.
    Given that $$c_U^V(z) = \frac{\rho_V(z)}{\rho_U(z)}\leq \frac{\rho_V(z)}{\rho_{\C - Q}(z)},$$ where $Q := P \cup \{z_n : n \in \N\}$, item (\ref{it: contraction}) follows from Proposition \ref{prop: inclusion of punctured spheres}.
\end{proof}

\subsection{Iteration on the unit disk} \label{subsec: iteration on the unit disk}

If $h\colon \D \to \D$ is a non-injective holomorphic map, the Denjoy-Wolff theorem \cite[Theorem 3.2.1]{Abate} states that any point $z \in \D$ converges under iteration of $h$ to a point $z_0 \in \overline{\D}$ that is independent of the initial choice of $z$. In this section, we explore holomorphic maps on the unit disk that satisfy stronger assumptions, allowing us to achieve more precise control on the behavior of their orbits. Many examples of such maps will appear in Section~\ref{sec: thurston theory}. Specifically, many pullback maps of (marked) Thurston maps with a marked set~$A$, where $|A| = 4$, satisfy these conditions.

We say that $z \in \widehat{\C}$ is a \textit{regular point} of a holomorphic map $g \colon U \to V$, where $U$ and $V$ are domains of $\widehat{\C}$, if either $z \in U$ and $\deg(g, z) = 1$, or $z$ is an isolated removable singularity of~$g$ and, after extending $g$ holomorphically to a neighborhood of $z$, $\deg(g, z) = 1$. A point $z \in \widehat{\C}$ is a \textit{fixed point} of the map $g \colon U \to V$ if either $z \in U$ and $g(z) = z$, or $z$ is an isolated removable singularity of $g$ and, after extending $g$ holomorphically to a neighborhood of $z$, we have $g(z) = z$. The concepts of repelling or attracting fixed points can be generalized in a similar way.

\begin{theorem}\label{thm: iteration on unit disk}
    Let $h \colon \D \to \D$ be a holomorphic map, and $\pi \colon \D \to V$ and $g \colon U \to V$ are holomorphic covering maps, where $U \subset V$ is a domain of $\widehat{\C}$ and $V = \widehat{\C} - P$ with $3 \leq |P| < \infty$. Suppose that $\pi(h(\D)) \subset U$ and $\pi = g \circ \pi \circ h$, i.e., the following diagram commutes:
    \begin{center}  
        \begin{tikzcd}
            \D \arrow[r, "h"] \arrow[d, "\pi"] & h(\D) \arrow [d, "\pi"]\\
            V & U \arrow[l, "g"]
        \end{tikzcd}
    \end{center}
    If the map $g \colon U \to V$ is non-injective, then exactly one of the following two possibilities is~satisfied:
    \begin{enumerate}
        \item for every $z \in \D$, the $h$-orbit of $z$ converges to the unique fixed point of $h$, or

        \item the sequence $(\pi(h^{\circ n}(z)))$ converges to the same repelling fixed point $x \in P$ of the map $g$, regardless of the choice of $z \in \D$.
    \end{enumerate}
\end{theorem}

    \begin{proof}
    Given that the maps $\pi \colon \D \to V$ and $g \circ \pi|\widetilde{U}\colon \widetilde{U} \to V$ are covering maps, where $\widetilde{U} := \pi^{-1}(U)$, it follows \cite[Section 1.3, Exercise 16]{hatcher} that $h\colon \D \to \widetilde{U}$ is also a covering map and, in particular, $\widetilde{U}$ is connected. Therefore, according to Schwarz-Pick's lemma, we have $\|\mathrm{D}h(z)\|_{\D} = c_{\widetilde{U}}^{\D}(h(z)) = c_U^V(\pi(h(z)))$ for every $z \in \D$. At the same time, Great Picard's theorem and the Riemann-Hurwitz formula \cite[Appendix~A.3]{Hubbard_Book_1} imply that $V - U$ contains at least one point. Hence, the inclusion $I \colon U \hookrightarrow V$ is locally uniformly contracting with respect to the hyperbolic metrics on $U$ and~$V$. As a result, $h$ is locally uniformly contracting with respect to the hyperbolic metric on~$\D$. In particular, $h$ has at most one fixed point, and if there exists a point $z \in \D$ such that its orbit $(h^{\circ n}(z))$ converges in $\D$, then every orbit of $h$ converges to the unique fixed point of $h$.

    Let us pick an arbitrary point $z_0 \in \D$. Define $z_n = h^{\circ n}(z_0)$ and $x_n = \pi(h^{\circ n}(z_0))$ for $n \geq 0$. Connect the points $z_0$ and $z_1$ by the hyperbolic geodesic~$\delta_0 \subset \D$. We denote by $\delta_n \subset \D$ the curve $h^{\circ n}(\delta)$ that connects $z_n$ and $z_{n + 1}$. By Schwarz-Pick's lemma, the sequence $(\ell_{\D}(\delta_n))$ is non-increasing. In particular, if $d_0 = d_{\D}(z_0, z_1)$, then $d_{\D}(z_n, z_{n + 1}) \leq \ell_{\D}(\delta_n) \leq d_0$ and $d_{V}(x_n, x_{n + 1}) \leq \ell_{V}(\pi(\delta_n)) \leq d_0$ for every~$n \geq 0$.

    Further, we structure the proof as a series of claims.

    \begin{claim1}
        If there exists a compact set $K \subset V$ such that $x_n \in K$ for infinitely many $n$, then the sequence $(z_n)$ converges in $\D$.
    \end{claim1}

    \begin{subproof}
        Since $\ell_{V}(\pi(\delta_n)) \leq d_0$, we can enlarge $K$ so that $\pi(\delta_n) \subset K$ for infinitely many~$n$. First, we demonstrate that there exists $\lambda < 1$ so that $c_U^V(z) \leq \lambda$ for all $z \in K \cap U$. Indeed, as we mentioned earlier, there exists a point $w \in U - V$. Therefore, for any $z \in K$, the distance $d_V(z, V - U) \leq d_V(z, w)$ is uniformly bounded from above. Hence, Proposition \ref{prop: mcmullens bound} and Remark \ref{rem: mcmullens bound} imply that such $\lambda$ exists. Finally, from the previous discussions, if $\pi(\delta_{n + 1}) \subset K$, then $\ell_{\D}(\delta_{n + 1}) \leq \lambda \ell_{\D}(\delta_n)$. In particular, this shows that $(\ell_{\D}(\delta_n))$ converges to~$0$ as~$n \to \infty$.

        There exists a subsequence of $(x_n)$ that converges to a point $x \in K$. By enlarging $K$, we can assume that $B_{V}(x, 2r) \subset K$ for a certain $r > 0$. Now, we choose $m$ such that $\ell_{\D}(\delta_m) < r(1 - \lambda) / 2$ and $d_{V}(x_m, x) < r$. Notice that $m$ is chosen so that $\ell_{\D}(\delta_m) \sum_{i = 0}^\infty \lambda^i  < r / 2$. Using this fact and applying induction (see \cite[Proof of Theorem 2.3, p. 20]{Nikita} for a similar argument), it can be shown that $\pi(\delta_{n}) \subset B_V(x, r)$ and $\ell_V(\pi(\delta_{n + 1})) \leq \lambda \ell_V(\pi(\delta_n))$ for all $n \geq m$. In other words, the distance between $z_n$ and $z_{n + 1}$ decreases exponentially, and the convergence follows from the completeness of the hyperbolic metric on $\D$.
    \end{subproof}

    \begin{claim2}
        If the sequence $(x_n)$ converges to $x \in P$ along some subsequence, then the entire sequence $(x_n)$ converges to $x$.
    \end{claim2}

    \begin{subproof}
    For any $x \in P$, we choose a neighborhood $V_x \subset \widehat{\C}$ so that if $x' \neq x''$, then $d_{V}(y_1, y_2) >~d_0$ for any $y_1 \in V_{x'} - P$ and $y_2 \in V_{x''} - P$. Now, suppose that the sequence $(x_n)$ has subsequences that converge to different limits $y' \in \widehat{\C}$ and $y'' \in \widehat{\C}$, respectively. 
    If either of $y'$ or $y''$ does not belong to $P$, then by Claim 1, the sequence $(z_n)$ must converge to a limit in $\D$, which leads to a contradiction.
    If we instead assume that $y', y'' \in P$, then it follows that $x_n \in V - \bigcup_{x \in P} V_x$ for infinitely many $n$. Once again, Claim~1 implies that $(z_n)$ must converge to a limit in~$\D$, and it also leads to a contradiction. Thus, the entire sequence $(x_n)$ converges to $x \in P$ as~$n \to \infty$.
    \end{subproof}

    \begin{claim3}
        If the sequence $(x_n)$ converges to $x \in P$, then $x$ is a fixed point of the map $g \colon U \to V$.
    \end{claim3}

    \begin{subproof}
        Suppose $x \in P$ is not an isolated removable singularity of $g$. Observe that for sufficiently large $n$, $\pi(\delta_n)$ lies within any given neighborhood of $x$. Therefore, based on Proposition~\ref{prop: coverings and inclusions} and the fact that $\|\mathrm{D}h(z)\|_{\D} = c_U^V(\pi(h(z)))$ for all $z \in \D$, we have $\ell_{\D}(\delta_{n + 1}) \leq \lambda \ell_{\D}(\delta_n)$ for $n$ large enough and some $\lambda < 1$. This means that the distance between $z_n$ and $z_{n + 1}$ decreases exponentially, so the sequence $(z_n)$ must have a limit in $\D$, leading to a contradiction. Therefore, $x \in P$ is an isolated removable singularity of $g$. Finally, since $x_n = g(x_{n + 1})$ for all $n \geq 0$, $x$ must be a fixed point of the map $g$.
    \end{subproof}

    \begin{claim4}
        If the sequence $(x_n)$ converges to a fixed point $x \in P$ of the map $g$, then $x$ is a repelling fixed point of $g$.
    \end{claim4}

    \begin{subproof}
        Without loss of generality, we assume $x = 0$.
        Let us choose a continuous parametrization $\delta_0\colon \I \to \D$ for the arc $\delta_0$. We then define a continuous curve $\delta \colon [0, +\infty) \to \D$ by setting $\delta(t + n) = h^{\circ n}(\delta_0(t))$ for any non-negative integer $n$ and $t \in \I$. It is straightforward to see that $\pi(\delta(t))$ converges to $0$ as $t$ tends to $\infty$, since $\pi(\delta(n)) = x_n \to x = 0$ as $n \to \infty$ and $\ell_V(\pi(\delta([n, n + 1]))) \leq d_0$ for all $n\geq 0$. Furthermore, note that $\pi(\delta(t)) = g(\pi(\delta(t + 1))$ for $t \geq 0$. Applying the Snail Lemma \cite[Lemma 13.2, Corollorary 13.3]{Milnor} to the curve $\pi \circ \delta$ and the map $g$ in a neighborhood of $0$, we obtain that either $|g'(0)| > 1$ or $g'(0) = 1$.

        Suppose $g'(0) = 1$. Further, we assume that $g$ is extended holomorphically to a neighborhood $0$. If we choose an arbitrary simply connected domain $D \subset \widehat{\C}$ such that $D \cap P = \{0\}$, then $g^{-1}(D)$ has a connected component $D'$ containing~0. Moreover, $D' \subset U \cup \{0\}$ is simply connected and $g|D'\colon D' \to D$ is a biholomorphism that fixes~$0$. Similarly, we can define the local inverse branches $\varphi_n \colon D \to U \cup \{0\}$ of $g^{\circ n}$ in a neighborhood of~$0$. In other words, $\varphi_n := (g^{\circ n}|D_n)^{-1}$, where $D_n$ is the connected component of 0 in~$g^{-n}(D)$. In particular, $\varphi_n$ is a biholomorphism, $\varphi_n(0) = 0$,~and~$(\varphi_n)'(0) = 1/(g^{\circ n})'(0) = 1$~for~every~$n \in \N$.
        
        Due to the previous discussions, $U \cup \{0\}$ does not contain at least three points of $\widehat{\C}$. Hence, Montel's theorem implies that the family $(\varphi_n)$ is normal. Therefore, up to a subsequence, it converges locally uniformly on~$D$ to some holomorphic map $\varphi$ such that $\varphi(0) = 0$ and $\varphi'(0) = 1$. Since $\varphi'(0) = 1$, $\varphi$ is injective in a neighborhood of 0.
        Thus, the iterates $(g^{\circ n})$ converge uniformly in a neighborhood of 0 to $\varphi^{-1}$, up to a subsequence. However, if
        $$
            g(z) = z + a_kz^k + \dots,
        $$
        where $a_k \neq 0$, $k \geq 2$, and three dots represent higher order terms, then
        $$
            g^{\circ n}(z) = z + na_kz^k + \dots,
        $$
        so $(g^{\circ n})^{(k)}(0)$ diverges as $n \to \infty$. This leads to a contradiction ruling out the possibility $g'(0) = 1$. Thus, $x = 0$ must be a repelling fixed point of the map $g$.
    \end{subproof}

    Claims 1-4 imply that if the sequence $(h^{\circ n}(z_0))$ diverges in $\D$, then the sequence $(\pi(h^{\circ n}(z_0)))$ converges to a repelling fixed point $x \in P$ of the map $g$. Furthermore, in this case $(\pi(h^{\circ n}(z)))$ converges to $x$ for every $z \in \D$. Indeed, this follows easily since $d_{V}(\pi(h^{\circ n}(z_0)), \pi(h^{\circ n}(z)))$ is bounded from above by $d_{\D}(z_0, z)$ due to Schwarz-Pick's lemma.    
\end{proof}

\section{Thurston theory} \label{sec: thurston theory}

In this section, we focus on the study of (marked) Thurston maps $f \colon (S^2, A) \rto$ that satisfy the following two conditions:
\begin{enumerate}[label=\normalfont{(\Roman*)}]
    \item \label{it: set A} the marked set $A$ contains exactly four points, and

    \item \label{it: set B} there exists a set $B \subset A$ such that $|B| = 3$, $S_f \subset B$, and $|\overline{f^{-1}(B)} \cap A| = 3$.
\end{enumerate}

It is evident that when the marked set $A$ coincides with the postsingular set $P_f$, conditions \ref{it: set A} and \ref{it: set B} are equivalent to conditions \ref{it: small s set}--\ref{it: extra property} from Section \ref{subsec: main results}. Also, it is worth noting that the case when $|\overline{f^{-1}(B)} \cap A| = 4$ or, equivalently, $A \subset \overline{f^{-1}(B)}$, is rather trivial due to Remark \ref{rem: constant sigma map}.

Using the tools developed Section \ref{sec: hyperbolic tools} and the theory of iteration of meromorphic functions, we analyze the corresponding pullbacks map defined on the one-complex dimensional Teichm\"uller space. It allows us to derived several properties of the corresponding Thurston maps and their Hurwitz classes. In particular, in Section~\ref{subsec: characterization problem} we prove Main Theorem~\ref{mainthm A} (see Theorem~\ref{thm: A}), and in Section~\ref{subsec: hurwitz classes} we prove Main Theorem~\ref{mainthm B} (see Theorem~\ref{thm: B}) and Corollary~\ref{corr: intro}.

Let $A = \{a_1, a_2, a_3, a_4\}$, $B = \{a_{i_1}, a_{i_2}, a_{i_3}\}$, and $C := \overline{f^{-1}(B)} \cap A = \{a_{j_1}, a_{j_2}, a_{j_3}\}$, where $i_1 < i_2 < i_3$ and $j_1 < j_2 < j_3$. Additionally, assume that $i$ and $j$ are the indices so that $a_i \in A - B$ and $a_j \in A - C$. Under this conventions, we have that $f(a_j) = a_i$ and $\deg(f, a_j) = 1$.

It is important to note that there may be multiple choices for the set $B$. However, when $|S_f| = 3$, the set $B$ is uniquely determined by the properties described in condition \ref{it: set B}. Also, another choice that we made, which will be relevant in our further arguments, is the indexing for the set $A$.

Now we are ready to introduce the following objects:
\begin{itemize}
    \item the map $\pi_B \colon \T_A \to \Sigma$, where $\Sigma = \widehat{\C} - \{0, 1, \infty\}$, is defined by $\pi_B([\varphi]) = \varphi(a_i)$, where the representative $\varphi \colon S^2 \to~\widehat{\C}$ is chosen so that $\varphi(a_{i_1}) = 0$, $\varphi(a_{i_2}) = 1$, and $\varphi(a_{i_3}) = \infty$;

    \item the map $\pi_C \colon \T_A \to \Sigma$ is defined by $\pi_C([\varphi]) = \varphi(a_j)$, where the representative $\varphi \colon S^2 \to~\widehat{\C}$ is chosen so that $\varphi(a_{j_1}) = 0$, $\varphi(a_{j_2}) = 1$, and $\varphi(a_{j_3}) = \infty$;

    \item the map $\omega_f\colon \T_A \to \Sigma$ is defined by $\omega_f = \pi_C \circ \sigma_f$;

    \item the map $F_f \colon \widehat{\C} \dto \widehat{\C}$ is a unique (which follows from Proposition \ref{prop: dependence}) holomorphic map such that $F_f = \varphi \circ f \circ \psi^{-1}$, where $\varphi, \psi \colon S^2 \to \widehat{\C}$ are orientation-preserving homeomorphisms satisfying $\varphi(a_{i_1}) = \psi(a_{j_1}) =~0$, $\varphi(a_{i_2}) = \psi(a_{j_2}) = 1$, and $\varphi(a_{i_3}) = \psi(a_{j_3}) = \infty$;

    \item the map $M_{i,j} \colon \Sigma \to \Sigma$ is defined as $h_i \circ h_j^{-1}$, where $h_i \colon \M_A \to \Sigma$ is defined by $h_i([\eta]) = \eta(a_i)$, with $\eta \colon A \to \widehat{\C}$ chosen so that $\eta(a_{i_1}) = 0$, $\eta(a_{i_2}) = 1$, $\eta(a_{i_3}) = \infty$, and the map $h_j \colon \M_A \to \Sigma$ is defined analogously.

    \item the set $W_f := \widehat{\C} - \overline{F_f^{-1}(\{0, 1, \infty\})}$ is a domain of $\widehat{\C}$;

    \item the map $G_f := F_f \circ M_{i, j}^{-1} \colon M_{i, j}(W_f) \to \Sigma$.
\end{itemize}

Of course, the maps $F_f$ and $G_f$, as well as the domain $W_f$, also depend on the choice of the subset $B$, and many other objects defined above depend on the indexing of $A$. However, for the simplicity, we are going to exclude these dependencies from the notation. Throughout this section, we will maintain the notation and conventions established above. Specifically, if $\widetilde{f} \colon (S^2, A) \rto$ is any Thurston map that is Hurwitz equivalent to $f \colon (S^2, A) \rto$, then we use the same set $B$ and the same indexing of the set $A$ when we work with the map $\widetilde{f}$ as we~do~it~with~$f$.

\begin{proposition}\label{prop: 1-dim pullback maps}
    The objects introduced above satisfy the following properties:
    \begin{enumerate}
        \item\label{it: diagram} diagram \eqref{fig: first fund diag} commutes.
        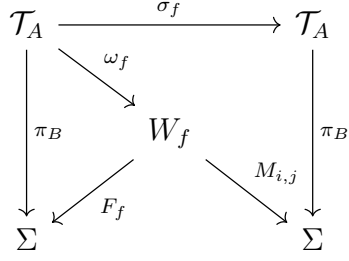
\begin{figure}[h]
        \begin{tikzcd}
            \T_A \arrow[rr, "\sigma_f"] \arrow[dd, "\pi_B"] \arrow[dr, "\omega_f"] & & \T_A \arrow [dd, "\pi_B"]\\
            & W_f \arrow[dl, "F_f"] \arrow[dr, "M_{i, j}"] & \\
            \Sigma  & & \Sigma            \end{tikzcd} 
        \caption{Fundamental diagram for Thurston maps satisfying conditions~\ref{it: set A} and~\ref{it: set B}.}\label{fig: first fund diag}
        \end{figure}
        \item\label{it: map F} $S_{F_f} \subset \{0, 1, \infty\}$ and $F_f \colon W_f \to \Sigma$ is a covering map;
        
        \item\label{it: coverings} $\pi_B \colon \T_A \to \Sigma$, $\pi_C \colon \T_A \to \Sigma$, and $\omega_f \colon \T_A \to W_f$ are holomorphic covering~maps;

        \item\label{it: map M} $M_{i, j}$ extends to a M\"obius transformation such that $\{0, 1, \infty\}$ is an $M_{i, j}$-invariant~subset;

        \item\label{it: holomorphicicty of sigma} $\sigma_f(\T_A) \subsetneq \T_A$ is open and dense in $\T_A$ and $\sigma_f \colon \T_A \to \sigma_f(\T_A)$ is a holomorphic covering~map;

        \item\label{it: hurwitz class} the maps $F_f$ and $G_f$, as well as the domain $W_f$, depend only on the Hurwitz equivalence class of the Thurston map $f$.
\end{enumerate}
\end{proposition}

\begin{proof}
    Let $\tau = [\varphi] \in \T_A$ and $\sigma_f(\tau) = [\psi] \in \T_A$, where the representatives $\varphi, \psi \colon S^2 \to \widehat{\C}$ are chosen so that $\varphi(a_{i_1}) = \psi(a_{j_1}) = 0$, $\varphi(a_{i_2}) = \psi(a_{j_2}) = 1$, $\varphi(a_{i_3}) = \psi(a_{j_3}) = \infty$, and the map $\varphi \circ f \circ \psi^{-1} \colon \widehat{\C} \dto \widehat{\C}$ is holomorphic. In particular, $\varphi \circ f \circ \psi^{-1}$ coincides with the map~$F_f$. Since, $\omega_f([\varphi]) = \pi_C([\psi]) = \psi(a_j)$, we have the following 
    $$(F_f \circ \omega_f)([\varphi]) = F_f(\psi(a_j)) = \varphi(f(a_j)) = \varphi(a_i) = \pi_B({\varphi}).$$ Thus, $\pi_B = F_f \circ \omega_f$. At the same time, it is straightforward to verify that $\pi_B = M_{i, j} \circ \pi_C$. Finally, $M_{i, j} \circ \omega_f = M_{i, j} \circ \pi_C \circ \sigma_f = \pi_B \circ \sigma_f$, and item (\ref{it: diagram}) follows.

    Item (\ref{it: map F}) directly follows from the definition of a singular set. Maps $\pi_B$ and $\pi_C$ are holomorphic coverings, as discussed in Section \ref{subsec: teichmuller spaces}. The map $\omega_f \colon \T_A \to W_f$ is also a holomorphic covering since both $\pi_B \colon \T_A \to \Sigma$ and $F_f \colon W_f \to \Sigma$ are holomorphic covering maps, and $\pi_B = F_f \circ \omega_f$ (see \cite[Section 1.3, Excercise 16]{hatcher}). Hence, item (\ref{it: coverings}) follows.

    The discussion of Section \ref{subsec: teichmuller spaces} shows that the maps $h_i, h_j \colon \M_A \to \Sigma$ are holomorphic. Consequently, $M_{i, j}$ is a conformal automorphism of $\Sigma$ that extends to a M\"obius transformation of $\widehat{\C}$ permuting 0, 1, and $\infty$. Alternative way to prove item (\ref{it: map M}) is through direct computation. For example, when $i = j$, then $M_{i, j} = \id_{\widehat{\C}}$; if $i = 1$ and $j = 2$, then $M_{1, 2}(z) = z / (z - 1)$,~and~so~on.

    Since $M_{i, j} \circ \omega_f = \pi_B \circ \sigma_f$, where $\omega_f$ and $\pi_B$ are holomorphic covering maps and $M_{i, j}$ is a M\"obius transformation, it follows that $\sigma_f \colon \T_A \to \sigma_f(\T_A)$ is also a holomorphic covering map, where $\sigma_f(\T_A) = \pi_B^{-1}(M_{i, j}(W_f))$. Note that by Great Picard's Theorem and the Riemann-Hurwitz formula \cite[Appendix A.3]{Hubbard_Book_1}, the set $\widehat{\C} - W_f$ contains at least one point different from the points $\{0, 1, \infty\}$. Therefore, $\sigma_f(\T_A)$ is different from $\T_A$. Since $W_f$ is open and dense in $\widehat{\C}$, it follows that $\sigma_f(\T_A)$ is open and dense in $\T_A$, establishing item \eqref{it: holomorphicicty of sigma}.
    

    Finally, let $\widetilde{f}$ be a Thurston map Hurwitz equivalent rel.\ $A$ to $f$. Suppose that $\phi_1 \circ \widetilde{f} = f \circ \phi_2$, where $\phi_1, \phi_2 \in \Homeo_0^+(S^2, A)$. Then $F_f = \widetilde{\varphi} \circ \widetilde{f} \circ \widetilde{\psi}^{-1}$, where $\widetilde{\varphi} = \varphi \circ \phi_1$ and $\widetilde{\psi} = \psi \circ \phi_2$. In particular, $\widetilde{\varphi}(a_{i_1}) = \widetilde{\psi}(a_{j_1}) = 0$, $\widetilde{\varphi}(a_{i_2}) = \widetilde{\psi}(a_{j_2}) = 1$, and $\widetilde{\varphi}(a_{i_3}) = \widetilde{\psi}(a_{j_3}) = \infty$. Therefore, $F_f$ and $F_{\widetilde{f}}$ coincide, as well as $W_f$ and $W_{\widetilde{f}}$, and item (\ref{it: hurwitz class}) follows.
\end{proof}    

\begin{remark}\label{rem: comparision with finite degree case}
    Proposition \ref{prop: 1-dim pullback maps} shows that for Thurston maps satisfying conditions \ref{it: set A} and~\ref{it: set B}, the corresponding pullback maps admit a commutative diagram analogous to diagram~\eqref{fig: fund diag} from Section \ref{subsec: intro pullback maps}. In this context, the M\"obius transformation $M_{i, j}$ serves the role of the $X$-map, the map $F_f$ takes the place of the $Y$-map, and the domain $W_f$ is an analog of the Hurwitz space. In particular, $\sigma_f$ has the ``$g$-map'' $G_f = F_f \circ M_{i, j}^{-1} \colon M_{i, j}(W_f) \to \Sigma$.
    
    In contrast to the finite degree case, when $f\colon (S^2, A) \rto$ is transcendental, the map $F_f \colon W_f \to \Sigma$ is a covering of infinite degree and $W_f$ is a countably punctured Riemann~sphere.
\end{remark}

\begin{remark}\label{rem: omitted values}
    Suppose that we are in the setting of Proposition \ref{prop: 1-dim pullback maps}. We have observed that $\sigma_f(\T_A) = \pi_B^{-1}(M_{i, j}(W_f))$ and
    the complement of the set $M_{i, j}(W_f)$ in $\Sigma$ contains at least one point. Given that $\pi_B \colon \T_A \to \Sigma$ is a covering map of infinite degree, the pullback map~$\sigma_f$ has infinitely many omitted values, i.e., the points of $\T_A - \sigma_f(\T_A)$, and they are not compactly contained in $\T_A$. Moreover, the set of omitted values of $\sigma_f$ is not discrete in the Teichm\"uller space $\T_A$ if the essential singularity of the map $f$ lies within the set $S^2 - A$.
\end{remark}

\subsection{Characterization problem} \label{subsec: characterization problem} In this section, we present and prove a slightly stronger version of Main Theorem \ref{mainthm A} utilizing the tools developed in Sections \ref{subsec: levy cycles} and \ref{subsec: iteration on the unit disk} along with the properties of pullback maps established in Proposition \ref{prop: 1-dim pullback maps}.

We recall that the Teichm\"uller space $\T_A$, where $|A| = 4$, is biholomorphic to the unit disk~$\D$, and the metric $d_T$ defined in Section \ref{subsec: teichmuller spaces} coincides with the hyperbolic metric $d_{\D}$. If $f\colon (S^2, A) \rto$ is a Thurston map satisfying assumptions \ref{it: set A} and~\ref{it: set B}, then the the corresponding pullback map $\sigma_f \colon \T_A \to \T_A$ is holomorphic. It can be established in two ways: either through the general approach outlined in Proposition \ref{prop: propeties of pullback map} or by more elementary methods as in item \eqref{it: holomorphicicty of sigma} of~Proposition~\ref{prop: 1-dim pullback maps}. It is worth mentioning that item \eqref{it: holomorphicicty of sigma} of Proposition \ref{prop: 1-dim pullback maps} and Schwarz-Pick's lemma imply that $\sigma_f$ is distance-decreasing on the Teichm\"uller space $\T_A$ (cf. Remark \ref{rem: further properties}).

\begin{theorem}\label{thm: A}
    Let $f \colon (S^2, A) \rto$ be a Thurston map of finite or infinite degree that satisfies properties \ref{it: set A} and~\ref{it: set B}. Then $f$ is realized rel.\ $A$ if and only if it has no weakly degenerate Levy fixed curve. Moreover,
    \begin{enumerate}
        \item \label{it: unique conjugacy} if $f$ is realized rel.\ $A$ by postsingularly finite holomorphic maps $g_1 \colon (\widehat{\C}, P_1) \rto$ and $g_2 \colon (\widehat{\C}, P_2)\rto$, then $g_1$ and $g_2$ are conjugate by a M\"obius transformation $M$, i.e., $M \circ g_1 = g_2 \circ M$, such that $M(P_1) = P_2$;
        \item \label{it: unique Levy cycle} if $f$ is obstructed rel.\ $A$, then it has a unique Levy fixed curve up to~homotopy~in~$S^2 -~A$.
    \end{enumerate}
\end{theorem}

\begin{proof}
Suppose that the Thurston map $f \colon (S^2, A) \rto$ is realized. According to Proposition~\ref{prop: levy cycles are obstructions}, $f\colon (S^2, A) \rto$ cannot have a Levy fixed curve. From Proposition \ref{prop: fixed point of sigma}, it follows that $\sigma_f$ has a fixed point in the Teichm\"uller space $\T_A$. As it was discussed previously, $\sigma_f$ is distance-decreasing on $\T_A$, which implies that it has a unique fixed point. Now, it is straightforward to verify using Proposition \ref{prop: def of sigma map} that item \eqref{it: unique conjugacy} holds (cf. \cite[Corollary~10.7.8]{Hubbard_Book_2} and \cite[Proposition 2.26]{our_approx}). 

Now, suppose that the Thurston map $f \colon (S^2, A) \rto$ is obstructed. Choose an arbitrary point $\tau_0 \in \T_A$ and define $\tau_n := \sigma_f^{\circ n}(\tau_0)$, $x_n := \pi_B(\tau_n)$, and $y_n := \pi_C(\tau_{n + 1})$ for every $n \geq 0$. Let $\tau_n = [\varphi_n] \in \T_A$, where the representative $\varphi_n \colon S^2 \to \widehat{\C}$ is chosen so that $\varphi_n(a_{i_1}) = 0$, $\varphi_n(a_{i_2}) = 1$, and $\varphi_n(a_{i_3}) = \infty$. Denote by $\psi_n \colon S^2 \to \widehat{\C}$ the unique (due to Proposition \ref{prop: def of sigma map}) orientation-preserving homeomorphism so that $g_n := \varphi_n \circ f \circ \psi_n^{-1} \colon \widehat{\C} \dto \widehat{\C}$ is holomorphic and $\psi_n(a_{j_1}) = 0$, $\psi_n(a_{j_2}) = 1$, and $\psi_n(a_{j_3}) = \infty$. According to the definition of the map $F_f$, it must coincide with the map $g_n$ for every $n \geq 0$. Moreover, it follows that $\tau_{n + 1} = [\psi_n]$. From this, we observe that $\varphi_n(a_i) = x_n$, $\psi_n(a_j) = y_n$, and $F_f(y_n) = x_n$ for every $n \geq 0$.

According to item \eqref{it: diagram} of Proposition \ref{prop: 1-dim pullback maps}, we have $\pi_B(\sigma_f(\T_A)) \subset M_{i, j}(W_f)$ and $\pi_B = G_f \circ \pi_B \circ \sigma_f$. Since $G_f$ is not injective, items \eqref{it: map F}-\eqref{it: holomorphicicty of sigma} of Proposition \ref{prop: 1-dim pullback maps} allow us to apply Theorem \ref{thm: iteration on unit disk} to the pullback map $\sigma_f$. This shows that the sequence $(x_n)$ converges to a repelling fixed point $x \in \{0, 1, \infty\}$ of the map~$G_f$. Given that $x_{n + 1} = M_{i, j}(y_n)$, the sequence $(y_n)$ converges to a regular point $y \in \{0, 1, \infty\}$ of the map $F_f$ due to item \eqref{it: map M} of Proposition~\ref{prop: 1-dim pullback maps}. 

We assume that the map $F_f$ is extended holomorphically to a neighborhood of $y$. Then there exists a disk~$D \subset~\widehat{\C}$ such that $D \cap \{0, 1, \infty\} = \{y\}$ and $F_f$ is injective on~$D$. Consider another disk $D'$ such that $\overline{D'} \subset D$ and the annulus $D - \overline{D'}$ has modulus greater than $5\pi e^{d_0}/\ell^*$, where $d_0 = d_T(\tau_0, \tau_1)$ and $\ell^* = \log(3 + 2 \sqrt{2})$. Observe that $y_n \in D'$ for all $n$ large enough and, in particular, each connected component of $\widehat{\C} - (D - \overline{D'})$ contains two points of the set $\psi_n(A)$.

Finally, by Schwarz-Pick's lemma, we have $d_T(\tau_n, \tau_{n + 1}) \leq d_T(\tau_0, \tau_1)$ for every $n \geq 0$. The existence of a weakly degenerate Levy fixed curve for the Thurston map $f\colon (S^2, A) \rto$ then follows from Proposition \ref{prop: finding levy cycles}, applied to $\tau_n = [\varphi_n]$ and $\tau_{n + 1} = [\psi_{n}]$, where $n$ is taken sufficiently large, and the annulus $D - \overline{D'}$ as above. The uniqueness part follows from Proposition \ref{prop: uniqueness of levy cycles} since the sequence $(\pi(\sigma_f^{\circ n}(\tau_0)))$ clearly leaves every compact set of the moduli space~$\M_A$.
\end{proof}

\subsection{Hurwitz classes} \label{subsec: hurwitz classes} 
In Section \ref{subsec: characterization problem}, we demonstrated how Proposition \ref{prop: 1-dim pullback maps}, especially commutative diagram \eqref{fig: first fund diag}, can be helpful for studying Thurston maps that satisfy assumptions \ref{it: set A} and~\ref{it: set B}. In this section, we further develop this idea by showing the significance of the dynamical properties of the map $G_f$ in understanding the Hurwitz class $\Hurw_{f, A}$ of a Thurston map $f \colon (S^2, A) \rto$ that satisfies properties \ref{it: set A} and \ref{it: set B}. 
In particular, we prove Main Theorem~\ref{mainthm B} (see Theorem~\ref{thm: B}) and Corollary \ref{corr: intro}. However, before proceeding with their proofs, we present two propositions that relate the fixed points of the map $G_f$ to the Thurston maps in the Hurwitz class $\Hurw_{f, A}$, which are either obstructed or realized depending on the properties of the corresponding fixed point. 

\begin{proposition}\label{prop: fixed points and realized maps}
    Let $f \colon (S^2, A) \rto$ be a Thurston map that satisfies conditions \ref{it: set A} and~\ref{it: set B}. Suppose that $x \in \Sigma$ is a fixed point of the map $G_f$. Then there exists a homeomorphism $\phi \in \Homeo^+(S^2, A)$ such that the Thurston map $\widetilde{f} := \phi \circ f \colon (S^2, A) \rto$ is realized by a holomorphic map $g \colon (\widehat{\C}, P) \rto$, where $P = \{0, 1, \infty, x\}$.
\end{proposition}

\begin{proof}
    Proposition \ref{prop: 1-dim pullback maps} suggests that there exist points $\tau_0$ and $\tau_1$ in the Teichm\"uller space $\T_A$ so that $\sigma_f(\tau_0) = \tau_1$ and $\pi_B(\tau_0) = \pi_B(\tau_1) = x$. Indeed, choose an arbitrary point $\tau_1$ of $\pi^{-1}_B(x)$. Since $x \in M_{i, j}(W_f)$, $\sigma_f^{-1}(\tau_1)$ is non-empty, and moreover, $\pi_B(\sigma_f^{-1}(\tau_1)) = \{x\}$. Thus, we can take $\tau_0$ to be any point of $\sigma_f^{-1}(\tau_1)$.

    Let $\tau_0 = [\varphi]$ and $\tau_1 = [\psi]$, where the representatives $\varphi, \psi \colon S^2 \to \widehat{\C}$ are chosen such that $g := \varphi \circ f \circ \psi^{-1} \colon \widehat{\C} \dto \widehat{\C}$ is holomorphic, $\varphi(a_{i_1}) = \psi(a_{i_1}) = 0$, $\varphi(a_{i_2}) = \psi(a_{i_2}) = 1$, $\varphi(a_{i_3}) = \psi(a_{i_3}) = \infty$, and $\varphi(a_i) = \psi(a_i) = x$. It is straightforward to verify that $g \colon (\widehat{\C}, P) \rto$ is a postsingularly finite holomorphic map, where $P = \varphi(A) = \psi(A) = \{0, 1, \infty, x\}$.
    
    Now, define $\phi := \psi^{-1} \circ \varphi \in \Homeo^+(S^2, A)$ and $\widetilde{f} := \phi \circ f \colon (S^2, A) \rto$. It is easy to see that $\psi \circ \widetilde{f} \circ \psi^{-1} = g$. Therefore, the Thurston map $\widetilde{f}\colon (S^2, A) \rto$ is combinatorially equivalent to~$g \colon (\widehat{\C}, P) \rto$. 
\end{proof}

\begin{proposition}\label{prop: fixed points and obstructed maps}
    Let $f \colon (S^2, A) \rto$ be a Thurston map that satisfies conditions \ref{it: set A} and~\ref{it: set B}. If $G_f$ has a repelling fixed point $x \in \{0, 1, \infty\}$, then there exists a homeomorphism $\phi \in \Homeo^+(S^2, A)$ such that the Thurston map $\widetilde{f} := \phi \circ f \colon (S^2, A) \rto$ is obstructed. Moreover, $f$ is totally unobstructed rel.\ $A$ if and only if none of the points $0$, $1$, or $\infty$ is a repelling fixed point of the map $G_f$.
\end{proposition}

\begin{proof}
    Suppose that $x \in \{0, 1, \infty\}$ is a repelling fixed point of the map $G_f$. Assume that $G_f$ is extended holomorphically to a neighborhood $x$. Let $U \subset M_{i, j}(W_f) \cup \{x\}$ be a neighborhood of $x$ where $G_f$ is injective and $\overline{U} \subset G_f(U)$. Define the local inverse branch $g\colon G_f(U) \to U$ of $G_f$ at $x$, i.e., $g := (G_f|U)^{-1}$. Note that every orbit of $g$ converges to~$x$, since $g$ is uniformly distance-decreasing with respect to the hyperbolic metric on $U$ according to Schwarz-Pick's lemma, Proposition \ref{prop: mcmullens bound}, and Remark~\ref{rem: mcmullens bound}.
    
    \begin{claim}
        The distance $d_{\Sigma}(y, g(y))$ converges to $0$ as $y \in G_f(U)$ tends to $x$.
    \end{claim}

    \begin{subproof}
        Without loss of generality, assume that $x = 0$ and $y, g(y) \in \D$. Since $x = 0$ is a repelling fixed point of the map $G_f$, by choosing $y$ sufficiently small, we can ensure that $\lambda |y| \leq |g(y)| \leq |y|$ for some $\lambda$, where $0 < \lambda < 1$. Let $p \colon \H \to \D$ be the holomorphic universal covering defined as $p(z) = \exp(iz)$ for $z \in \H$. Define $y_1 := \arg(y) - i \log(|y|) \in p^{-1}(y)$ and $y_2 := \arg(g(y)) - i \log(|g(y)|) \in p^{-1}(g(y))$. Similarly to the proof of Proposition \ref{prop: inclusion of punctured spheres}, we have
        $$
            d_{\H}(y_1, y_2) \leq \log\left(\frac{\log|g(y)|}{\log|y|}\right) + \frac{2\pi}{|\log|y||} \leq \log\left(\frac{\log|y| + \log \lambda}{\log|y|}\right) + \frac{2\pi}{|\log|y||}.
        $$
        This shows that the distance $d_{\H}(y_1, y_2)$ converges to $0$ as $y$ approaches $x$. According to Schwarz-Pick's lemma, the same holds for $d_{\D}(y, g(y))$ and $d_{\Sigma}(y, g(y))$.
    \end{subproof}

    Therefore, by making $U$ even smaller, we can assume that $d_{\Sigma}(y, g(y)) < d_{\Sigma}(y, y')$, where $y \in G_f(U)$ is any point other than $x$ and $y' \in G_f^{-1}(y)$ with $y \neq g(y)$. 
    
    Now, we choose $x_1 \in U - \{x\}$ and let $x_0 = G_f(x_1)$. Similarly to the proof of Proposition~\ref{prop: fixed points and realized maps}, there exists two points $\tau_1$ and $\tau_2$ of the Teichm\"uller space $\T_A$ so that $\sigma_f(\tau_0) = \tau_1$, $\pi_B(\tau_0) = x_0$, and $\pi_B(\tau_1) = x_1$. Since $\pi_B \colon \T_A \to \Sigma$ is a holomorphic covering map, then there also exists a point $\tau \in \T_A$ so that $d_T(\tau, \tau_1) = d_{\Sigma}(x_0, x_1)$ and $\pi_B(\tau) = x_0$. 
    
    Let $\tau = [\varphi]$ and $\tau_0 = [\psi]$, where the representatives $\varphi, \psi \colon S^2 \to \widehat{\C}$ are chosen so that $\varphi|A = \psi|A$. Define the homeomorphism $\phi := \varphi^{-1} \circ \psi \in \Homeo^+(S^2, A)$. According to Remark~\ref{rem: functoriality}, $\sigma_{\phi}(\tau) = \tau_0$ and $\sigma_{\widetilde{f}}(\tau) = \sigma_{f}(\sigma_{\phi}(\tau)) = \sigma_f(\tau_0) = \tau_1$, where $\widetilde{f} := \phi \circ f \colon (S^2, A) \rto$ is a Thurston~map.

    Schwarz-Pick's lemma, along with items \eqref{it: coverings} and \eqref{it: holomorphicicty of sigma} of Proposition \ref{prop: 1-dim pullback maps}, implies that 
    $$
    d_\Sigma(\pi_B(\sigma_{\widetilde{f}}^{\circ n}(\tau)), \pi_B(\sigma_{\widetilde{f}}^{\circ (n + 1)}(\tau))) \leq d_T(\tau, \sigma_{\widetilde{f}}(\tau)) = d_T(\tau, \tau_1) = d_\Sigma(x_0, x_1).
    $$ 
    At the same time, it follows from items (\ref{it: diagram}) and (\ref{it: hurwitz class}) of Proposition~\ref{prop: 1-dim pullback maps} that 
    $$
        \pi_B(\sigma_{\widetilde{f}}^{\circ (n + 1)}(\tau)) \in G_f^{-1}(\pi_B(\sigma_{\widetilde{f}}^{\circ n}(\tau))).
    $$
    Since $\pi_B(\tau) = x_0 \in G_f(U)$, and based on the previous assumptions, we have $\pi_B(\sigma_{\widetilde{f}}^{\circ n}(\tau)) = g^{\circ n}(x_0)$. Thus, $\pi_B(\sigma_{\widetilde{f}}^{\circ n}(\tau))$ converges to $x$ as $n \to \infty$. Given that $\sigma_{\widetilde{f}}$ is 1-Lipschitz, it cannot have a fixed point. Hence, by Proposition~\ref{prop: fixed point of sigma}, the Thurston map $\widetilde{f} \colon (S^2, A) \rto$ must be~obstructed.

    Suppose none of points $0$, $1$, or $\infty$ is a repelling fixed point of the map $G_f$. Let $\widehat{f}$ be any Thurston map Hurwitz equivalent rel.\ $A$ to $f$. Then, according to item (\ref{it: hurwitz class}) of Proposition~\ref{prop: 1-dim pullback maps}, we have $W_f = W_{\widehat{f}}$ and $G_f = G_{\widehat{f}}$. Taking into account Proposition \ref{prop: 1-dim pullback maps} and applying Theorem~\ref{thm: iteration on unit disk}, we see that $\sigma_{\widehat{f}}$ must have a fixed point. Thus, it follows from Proposition~\ref{prop: fixed point of sigma} that $\widehat{f}$ is~realized~rel.\ ~$A$. 
\end{proof}

\begin{remark}\label{rem: fixed point of g_f}
    It is clear that $G_f$ extends to a postsingularly finite holomorphic map having at most one essential singularity and a postsingular set contained within $\{0, 1, \infty\}$. Therefore, by Lemma \ref{lemm: classification of fixed point}, every fixed point of $G_f$ is either superattracting or repelling. Furthermore, the only possible superattracting fixed points are $0$, $1$, and $\infty$.
\end{remark}

Now we are ready to state and prove a slightly stronger version of Main Theorem \ref{mainthm B}.

\begin{theorem}\label{thm: B}
    Let $f \colon (S^2, A) \rto$ be a Thurston map of finite or infinite degree that satisfies conditions \ref{it: set A} and~\ref{it: set B}. Then 
    \begin{enumerate}
        \item \label{it: totally unobstructed} $f$ is totally unobstructed rel.\ $A$ if and only if there are no two points $a, b \in A$ such that $\deg(f, a) = \deg(f, b) = 1$ and $f(\{a, b\})$ equals $\{a, b\}$ or $A - \{a, b\}$;

        \item \label{it: inf obstructed} if $f$ is not totally unobstructed rel.\ $A$, then its Hurwitz class $\Hurw_{f, A}$ contains infinitely many pairwise combinatorially inequivalent obstructed Thurston maps;

        \item \label{it: inf realized} if $f$ has infinite degree, then its Hurwitz class $\Hurw_{f, A}$ contains infinitely many pairwise combinatorially inequivalent realized Thurston maps.
    \end{enumerate}
\end{theorem}

\begin{proof}
    Without loss of generality, we assume that $A$ is indexed so that $B = \{a_1, a_2, a_3\}$, and therefore, $i = 4$ (see the beginning of Section \ref{sec: thurston theory}). We then analyze four different cases based on the value of $j, 1 \leq j \leq 4$, to find out when one of the points $0$, $1$, or $\infty$ is a repelling fixed point of the map $G_f$ (this analysis will be needed to apply Proposition \ref{prop: fixed points and obstructed maps}). We also recall that $\deg(f, a_j) = 1$ and $f(a_j) = a_i$.
    \begin{itemize}
        \item For the case $j = 1$, we have $M_{4, 1}(z) = 1/z$. Therefore,
        $$
            G_f(0) = F_f(\infty) = (\varphi \circ f \circ \psi^{-1})(\infty) = \varphi(f(a_4)).
        $$
        This means that $0$ is a fixed point of $G_f$ is and only if $f(a_4) = \varphi^{-1}(0) = a_1$. Furthermore, according to Remark \ref{rem: fixed point of g_f}, $0$ is a repelling fixed point of $G_f$ if and only if $f(a_4) = a_1$ and $a_4$ is a regular point of $f$. Similarly, $1$ is a repelling fixed point of~$G_f$ if and only if $f(a_3) = a_2$ and $\deg(f, a_3) = 1$. Lastly, $\infty$ is a repelling fixed point of the map~$G_f$ if and only if $f(a_2) = a_3$ and $\deg(f, a_2) = 1$. 

        \item For the case $j = 2$, we have $M_{4, 2}(z) = (z - 1) / z$. Similarly to the previous case, one of the points $0$, $1$, and $\infty$ is a repelling fixed point of $G_f$ if and only if $f(a_3) = a_1$ and $\deg(f, a_3) = 1$, or $f(a_4) = a_2$ and $\deg(f, a_4) = 1$, or $f(a_1) = a_3$ and $\deg(f, a_1) = 1$.

        \item For the case $j = 3$, we have $M_{4, 3}(z) = 1 - z$. Here, one of the points $0$, $1$, and $\infty$ is a repelling fixed point of $G_f$ if and only if $f(a_2) = a_1$ and $\deg(f, a_2) = 1$, or $f(a_1) = a_2$ and $\deg(f, a_1) = 1$, or $f(a_4) = a_3$ and $\deg(f, a_4) = 1$.

        \item For the case $j = 4$, we have $M_{4, 4} = \id_{\widehat{\C}}$. In this case, one of the points $0$, $1$, and $\infty$ is a repelling fixed point of $G_f$ if and only if $f(a_k) = a_k$ and $\deg(f, a_k) = 1$ for some $k = 1, 2, 3$.
    \end{itemize}

    Summarizing the calculations above and applying Proposition \ref{prop: fixed points and obstructed maps}, we obtain item (\ref{it: totally unobstructed}).

    To establish item (\ref{it: inf obstructed}), we largely follow the approach used in the proof of \cite[Theorem~9.2(V)]{Pullback_invariants}. Suppose that $f$ is not totally unobstructed rel.\ $A$, i.e., there exists an obstructed Thurston map in $\Hurw_{f, A}$. Without loss of generality, we assume that this map is $f \colon (S^2, A) \rto$. Theorem \ref{thm: A} shows that there exists a Levy fixed curve $\gamma$ for $f \colon (S^2, A) \rto$. Define $f_n := T_{\gamma}^{\circ n} \circ f \colon (S^2, A) \rto$, where $n \in \Z$ and $T_\gamma \in \Homeo^+(S^2, A)$ is the \textit{Dehn twist} about a curve~$\gamma$. Clearly, each Thurston map $f_n$ has a Levy fixed curve $\gamma$, and therefore, is obstructed rel.\ $A$ by Proposition \ref{prop: levy cycles are obstructions}. We will show that these Thurston maps are pairwise combinatorially inequivalent rel.\ $A$. 
    
    Suppose the contrary. Then there exist two homeomorphisms $\phi_1, \phi_2 \in \Homeo^+(S^2)$ such that $\phi_1(A) = \phi_2(A) = A$, $\phi_1$ is isotopic rel.\ $A$ to $\phi_2$, and $f_n = \phi_1 \circ f_m \circ \phi_2^{-1}$ for some $m \neq n$. 
    Therefore, $f_n\colon (S^2, A) \rto$ has $\phi_1(\gamma)$ as a Levy fixed curve. However, Theorem \ref{thm: A} states that this Levy fixed curve is unique up to homotopy in $S^2 - A$. This implies that $\gamma$ and $\phi_1(\gamma)$ are homotopic in $S^2 - A$, and thus, $\phi_1$ and $\phi_2$ are isotopic rel.\ $A$ to $T_\gamma^{\circ k}$ for some $k \in \Z$. 

    One can easily see that $f$ commutes with the Dehn twist $T_\gamma$ up to isotopy rel.\ $A$, meaning $T_\gamma \circ f$ is isotopic rel.\ $A$ to $f \circ T_\gamma$. Indeed, up to isotopy rel.\ $A$, we can assume that $f$ is the identity on a certain annulus in $S^2 - A$ with a core curve $\gamma$. Considering the previous discussion, we conclude that $f$ is isotopic rel.\ $A$ to $f \circ T_\gamma^{\circ (m - n)}$. The following claim proves that it is not possible, and item \eqref{it: inf obstructed} follows.

    \begin{claim}
        Suppose that $f$ is isotopic rel.\ $A$ to $f \circ \phi$, where $\phi \in \Homeo^+(S^2, A)$. Then $\phi$ is isotopic rel.\ $A$ to $\id_{S^2}$.
    \end{claim}

    \begin{subproof}
        According to Definition \ref{def: isotopy of thurston maps}, we can assume without loss of generality that $f = f \circ~\phi$. There exist orientation-preserving homeomorphisms $\varphi, \psi \colon \widehat{\C} \to S^2$ such that the map $g := \varphi \circ f \circ \psi^{-1} \colon \widehat{\C} \dto \widehat{\C}$ is holomorphic. One can easily check that $g = g \circ h$, where $h := \psi \circ \phi \circ \psi^{-1} \colon \widehat{\C} \to \widehat{\C}$. Since $h$ must be a M\"obius transformation and $h$ fixes the points of the set $\psi(A)$, it follows that $h = \id_{\widehat{\C}}$. Thus, $\phi$ is also the identity map, proving the claim.
    \end{subproof}

    Lemma \ref{lemm: inf many repelling fixed points} implies that the map $G_f$ has infinitely many fixed points when the map $f$ is transcendental. According to Proposition \ref{prop: fixed points and realized maps}, every such fixed point, apart from 0, 1, and $\infty$, corresponds to a realized Thurston map in $\Hurw_{f, A}$. However, some of these maps might be combinatorially equivalent rel.\ $A$. Nevertheless, we will show that only finitely many of them can be pairwise combinatorially equivalent rel.\ $A$.
    
    Consider two Thurston maps $f_1$ and $f_2$ realized rel.\ $A$ by postsingularly finite holomorphic maps $g_1 \colon (\widehat{\C}, P_1) \rto$ and $g_2 \colon (\widehat{\C}, P_2) \rto$, respectively, where $P_1 = \{0, 1, \infty, x_1\}$ and $P_2 = \{0, 1, \infty, x_2\}$, with $x_1$ and $x_2$ being distinct fixed points of the map $G_f$. If $f_1$ and~$f_2$ are combinatorially equivalent rel.\ $A$, then $f_1$ is realized rel.\ $A$ by both $g_1 \colon (\widehat{\C}, P_1) \rto$ and $g_2 \colon (\widehat{\C}, P_2) \rto$. By item \eqref{it: unique conjugacy} of Theorem~\ref{thm: A}, there exists of a M\"obius transformation $M$ such that $M \circ g_1 = g_2 \circ M$ and $M(P_1) = P_2$. In particular, $M$ is not the identity map and $\{0, 1, \infty\} \subset M(\{0, 1, \infty, x_1\})$. Since a M\"obius transformation is uniquely determined by its values at three distinct points of $\widehat{\C}$, there can be at most 24 such Thurston maps that are pairwise combinatorially equivalent rel.\ $A$, and item \eqref{it: inf realized} follows 
\end{proof}

Let us now proceed to prove Corollary \ref{corr: intro} from Section \ref{subsubsec: hurwitz}. First of all, we recall the definition of a \textit{parameter space}.

\begin{definition}\label{def: parameter space}
    Let $g \colon \C \to \widehat{\C}$ be a non-constant meromorphic map of finite type. Then the parameter space of $g$ is defined as follows:
    \begin{align*}
        \Par(g) := \{\varphi \circ g \circ \psi \colon \C \to \widehat{\C} \text{ holomorphic } \text{for some } \varphi \in \Homeo^+(\widehat{\C}) \text{ and } \psi \in \Homeo^+(\C)\}.
    \end{align*}
\end{definition}

\begin{proof}[Proof of Corollary \ref{corr: intro}]
    Note that $g$ should have at least two singular values and, according to Great Picard's Theorem, at most two \textit{exceptional} values, i.e., points $w \in \widehat{\C}$ such that the preimage $g^{-1}(w)$ is finite. Clearly, every exceptional value is an asymptotic value of $g$.
    Therefore, by post-composing $g$ with a M\"obius transformation, we can assume that $\{0, \infty\} \subset S_g \subset \{0, 1, \infty\}$ and $g^{-1}(1)$ is infinite. By pre-composing $g$ with an affine transformation, we can assume that $g(0) = 1$ and $g(1) = 1$. Let $x, y \neq 0, 1, \infty$ be two points in $\widehat{\C}$ so that $g(x) = y$.

    Next, choose four distinct point $a$, $b$, $c$, and $d$ in $S^2$, and two orientation-preserving homeomorphisms $\varphi, \psi \colon S^2 \to \widehat{\C}$ such that:
    \begin{itemize}
        \item $\varphi(a) = 0$, $\varphi(b) = 1$, $\varphi(c) = y$, and $\varphi(d) = \infty$, and
        \item $\psi(a) = 0$, $\psi(b) = x$, $\psi(c) = 1$, and $\psi(d) = \infty$.
    \end{itemize}

    Then $f := \varphi^{-1} \circ g \circ \psi \colon S^2 \dto S^2$ is a topologically holomorphic map. Moreover, $\{a, d\} \subset S_f \subset \{a, b, d\}$, and $f(a) = b$, $f(b) = c$, and $f(c) = b$, while $d \in S_f$ is the essential singularity of the map $f$. In other words, $f$ is a Thurston map with the postsingular set $P_f = \{a, b, c, d\}$. Moreover, it is easy to see that $f$ satisfies conditions \ref{it: set A} and \ref{it: set B}. Indeed, by setting $B = \{a, b, d\}$, we have $\overline{f^{-1}(B)} \cap P_f = \{a, c, d\}$.

    Item (\ref{it: inf realized}) of Theorem \ref{thm: B} implies that the Hurwitz class $\Hurw_f$ contains infinitely many realized Thurston maps that are pairwise combinatorially inequivalent. Clearly, each of these maps is realized by a postsingularly finite map from $\Par(g)$. Obviously, these maps must be pairwise (topologically or conformally) non-conjugate, leading to the desired result.
\end{proof}

\begin{remark}\label{rem: inf many psf entire maps}
    If $g \colon \C \to \C$ is a non-constant entire map of finite type, its \textit{entire parameter space} is defined by
    \begin{align*}
    \Par_E(g) := \{\varphi \circ g \circ \psi \colon \C \to \C \text{ holomorphic } \text{for some } \varphi, \psi \in \Homeo^+(\C)\}. 
\end{align*}
    Following the proof of Corollary \ref{corr: intro}, one can show that if $g$ is a transcendental and $|S_g| \leq~3$, then $\Par_E(g)$ contains infinitely many postsingularly finite entire maps with four postsingular values that are pairwise non-conjugate.
\end{remark}

\begin{remark}\label{rem: uncountably many}
    Using the framework of \textit{line complexes} (see \cite[Section XI]{Goldberg} or \cite[Section 2.7]{our_approx}), it can be shown that there are uncountably many distinct parameter spaces. Therefore, Corollary \ref{corr: intro} implies that conditions \ref{it: set A} and \ref{it: set B} are met by uncountably many pairwise combinatorially inequivalent realized Thurston maps with four postsingular values. Furthermore, by applying item \eqref{it: totally unobstructed} of Theorem \ref{thm: A}, it is easy to verify that the Thurston maps constructed in the proof Corollary \ref{corr: intro} are not totally unobstructed if $z = 1$ is a regular point of the map $g$. This shows that the family of Thurston maps under consideration also includes uncountably many pairwise combinatorially inequivalent obstructed Thurston maps. According to Remark \ref{rem: inf many psf entire maps}, the same observations hold even if we restrict to the class of entire Thurston maps.
\end{remark}

\subsection{Examples} \label{subsec: examples} In this section, we provide examples of several families of Thurston maps that satisfy conditions \ref{it: set A} and \ref{it: set B}. We also demonstrate how the framework of Sections \ref{subsec: characterization problem} and \ref{subsec: hurwitz classes} applies to these rather concrete cases. 

\begin{example}[Exponential maps]\label{ex: exponential}
Let $f \colon S^2 \dto S^2$ be an entire Thurston map with $P_f = \{a, b, c, d\}$ and $S_f = \{a, d\}$, where $d$ is the essential singularity of $f$. We recall that Thurston maps of this type are called \textit{exponential} Thurston maps. It is easy to see that  $a$ must be an omitted value of the map $f$. In particular, $f$ has one of the following two dynamical portraits on the set $\{a, b, c\}$ as illustrated in Figure \ref{fig: exp portraits}: either the singular value $a$ has pre-period 1 and period 2, or it has pre-period 2 and period 1.

\begin{figure}[h]
\begin{minipage}{0.3\textwidth}   
    \begin{flushleft}
    \begin{tikzcd}
        a \arrow[r] & b \arrow[r] & c \arrow[l, bend right = 50, swap]
    \end{tikzcd}
    \end{flushleft}
\end{minipage}
\begin{minipage}{0.3\textwidth}
\begin{center}
    \begin{flushright}
    \begin{tikzcd}
        a \arrow[r] & b \arrow[r] & c \arrow[loop,out=30,in=-30,distance=20]
    \end{tikzcd}
    \end{flushright}
\end{center}
\end{minipage}
\caption{Possible orbit of the singular value of an exponential Thurston map with four postsingular values.}\label{fig: exp portraits}
\end{figure}
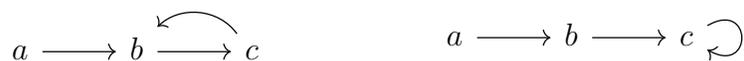

In both cases, the map $f$ satisfies properties \ref{it: set A} and \ref{it: set B}. In particular, Theorem \ref{thm: A} shows that $f$ is realized if and only if it has no Levy fixed curve, which must, in fact, be degenerate. This, in particular, provides a new proof for the more general result \cite[Theorem 2.4]{HSS} in the case of four postsingular values.

According to item (\ref{it: totally unobstructed}) of Theorem \ref{thm: B}, if the singular value $a$ has pre-period 2 and period~1, then $f$ it totally unobstructed.  However, if $a$ has pre-period~1 and period 2, the Thurston map $f$ is never totally unobstructed since $f(\{b, c\}) = \{b, c\}$ and $\deg(f, b) = \deg(f, c) = 1$ because otherwise either $b$ or $c$ would be a singular value of $f$.
Moreover, in this case the Hurwitz class $\Hurw_f$ contains infinitely many pairwise combinatorially inequivalent obstructed Thurston maps by item~\eqref{it: inf obstructed} of Theorem \ref{thm: B}. In both cases, item~(\ref{it: inf realized}) of Theorem \ref{thm: B} states that the Hurwitz class of $f$ contains infinitely many pairwise combinatorially inequivalent realized Thurston maps.

We further assume that $a = a_1$, $b = a_2$, $c = a_3$, and $d = a_4$, and adopt the notation introduced at the beginning of Section \ref{sec: thurston theory}. Our goal is to derive explicit formula for the map~$G_f$, analyze its dynamics, and observe the phenomena described in Propositions \ref{prop: fixed points and realized maps} and \ref{prop: fixed points and obstructed maps}, as well as in the proof of Theorem \ref{thm: B}.

First, we consider the case when the singular value $a$ has pre-period 1 and period 2. Let $B = \{a, b, d\} = \{a_1, a_2, a_4\}$, and then $C = \overline{f^{-1}(B)} \cap P_f = \{a, c, d\} = \{a_1, a_3, a_4\}$. In particular, here $i = 3$ and $j = 2$.

Let us compute the map $F_f$. It is evident that $F_f$ is a transcendental entire function. Moreover, $S_{F_f} = \{0, \infty\}$. By the classical theory of covering maps, $F_f(z) = \alpha \exp(\lambda z)$ for some $\alpha, \lambda \in \C^*$. Given that $F_f(0) = 1$ and $F_f(1) = 1$, it follows that $\alpha = 1$ and $\lambda = 2\pi i k, k \in \Z^*$, where $k$ is determined by the Hurwitz equivalence class of~$f$, as stated in item~\eqref{it: hurwitz class} of Proposition~\ref{prop: 1-dim pullback maps}. At the same time, $W_f = \C - \{l/k: l \in \Z\}$ and $M_{3, 2}(z) = 1/z$. Therefore, $G_f(z) = \exp(2\pi i k / z)$.

In particular, $0$ is the essential singularity of $G_f$, $1$ is a fixed repelling fixed point of $G_f$ of multiplier $2\pi i k$, and $G_f(\infty) = 1$. Moreover, by Lemma \ref{lemm: inf many repelling fixed points}, the map $G_f$ has infinitely many repelling fixed points. Thus, Propositions \ref{prop: fixed points and realized maps} and~\ref{prop: fixed points and obstructed maps} already imply that the Hurwitz class~$\Hurw_f$ contains both realized and obstructed Thurston maps.

Now, let the singular value $a$ has pre-period 2 and period 1. Let $B = \{a, c, d\} = \{a_1, a_3, a_4\}$ and then $C = \{b, c, d\} = \{a_2, a_3, a_4\}$. Here, $i = 2$ and $j = 1$.
Similarly to the previous case, we find that $F_f(z) = \exp(2\pi i k z)$ with $k \in \Z^*$, $W_f = \C - \{l/k: l \in \Z\}$, $M_{2, 1}(z) = z/(z - 1)$, and $G_f(z) = \exp(2\pi i k / (z - 1))$. In particular, $1$ is the essential singularity of $G_f$, and $G_f(0) = G_f(\infty) =  1$. Thus, none of the points $0$, $1$, or~$\infty$ is a fixed point of the map $G_f$. Therefore, Proposition~\ref{prop: fixed points and obstructed maps} implies that $f$ is indeed totally unobstructed.

\end{example}

\begin{example}[Entire maps with three singular values]\label{ex: entire}
    Let $f \colon S^2 \dto S^2$ be an entire Thurston map with the postsingular set $P_f = \{a, b, c, d\}$, where $S_f =\{a, b, d\}$ and $d$ is the essential singularity of $f$. If $f$ satisfies condition~\ref{it: set B}, then it should have one of the  three (up to relabeling) possible dynamical portraits on the set $\{a, b, c\}$ as illustrated in Figure \ref{fig: portraits}. Additionally, there are four more dynamical portraits when condition~\ref{it: set B} is not satisfied.

    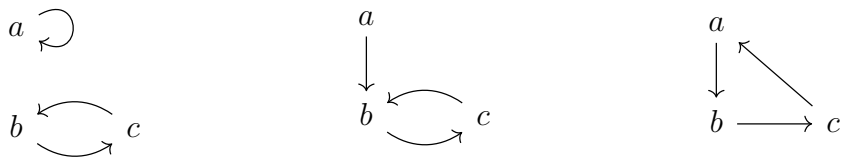
\begin{figure}[h]
    \begin{minipage}{0.225\textwidth}
        \begin{flushleft}
        	\begin{tikzcd}    
                a \arrow[loop,out=30,in=-30,distance=20,""]\\
                b \arrow[r, bend right = 35, swap] & c \arrow[l, bend right = 35, swap]
            \end{tikzcd}
        \end{flushleft}
    \end{minipage}
    \begin{minipage}{0.225\textwidth}
        \begin{center}
        \begin{tikzcd}
            a \arrow[d]\\
            b \arrow[r, bend right = 35, swap] & c \arrow[l, bend right = 35, swap]
        \end{tikzcd}
        \end{center}
    \end{minipage}
    \begin{minipage}{0.225\textwidth}
        \begin{flushright}
            \begin{tikzcd}
            a \arrow[d]\\
            b \arrow[r] & c \arrow[lu]
        \end{tikzcd}
        \end{flushright}
    \end{minipage}
    \caption{Possible dynamical portraits for an entire Thurston map with three singular and four postsingular values that satisfies condition \ref{it: set B}.}\label{fig: portraits}
    \end{figure}

    Theorem \ref{thm: A} states that a Thurston map with one of the dynamical portraits as in Figure~\ref{fig: portraits} is realized if and only if it has no weakly degenerate Levy fixed curve. Furthermore, according to item (\ref{it: totally unobstructed}) of Theorem \ref{thm: B}, $f$ is totally unobstructed for the first two dynamical portraits (from left to right) in Figure \ref{fig: portraits} if and only if $\deg(f, c) = 1$, and $f$ is always totally unobstructed for the third dynamical portrait.

    Let $a = a_1$, $b = a_2$, $c = a_3$, and $d = a_4$. Then we take $B = S_f = \{a, b, d\} = \{a_1, a_2, a_4\}$, and then $C = \{a, c, d\} = \{a_1, a_3, a_4\}$. In particular, $i=3$ and $j=2$. It can can be verified that $F_f$ is an entire function with $S_{F_f} = \{0, 1, \infty\}$. At the same time, $M_{3, 2}(z) = 1/z$, and therefore $G_f(z) = F_f(1/z)$. In particular, $0$ is the essential singularity of $G_f$. Thus, we have the following behavior of the map $G_f$ on the ``cusps'' 0 and 1 of the moduli space $\Sigma \sim \M_A$:
    \begin{itemize}
        \item for the first dynamical portrait, $G_f(\infty) = 0$ and $1$ is a fixed point of $G_f$ that, according to Lemma \ref{lemm: classification of fixed point}, is repelling if $\deg(f, c) = 1$, or superattracting otherwise;

        \item for the second dynamical portrait, $G_f(\infty) = 1$ and $1$ is a fixed point of $G_f$ that is repelling if and only if $\deg(f, c) = 1$, or superattracting otherwise;

        \item for the third dynamical portrait, $G_f(1) = 0$ and $G_f(\infty) = 1$. In particular, neither of $0$, $1$, and $\infty$ is a fixed point of the map $G_f$.
    \end{itemize}
\end{example}

\begin{example}[Non-entire examples]\label{ex: non-entire}
    Most of the observations in Example \ref{ex: entire} do not depend on the condition $f^{-1}(d) = \emptyset$, i.e., that the Thurston map $f$ is entire. Furthermore, there are more non-entire examples of Thurston maps that satisfy conditions \ref{it: set A} and \ref{it: set B}. For instance, if $f\colon S^2 \dto S^2$ is a Thurston map with $|S_f| \leq 3$, where the postsingular set $P_f = \{a, b, c, d\}$ does not contain an essential singularity (e.g., $f$ could be a finite degree map), and $f$ has one of the postsingular portraits shown in Figure \ref{fig: non entire portraits}. In particular, Theorem \ref{thm: A} provides a novel proof of celebrated Thurston's characterization theorem \cite[Theorem 1]{DH_Th_char} for a specific family of finite degree Thurston maps with four postcritical~values.

    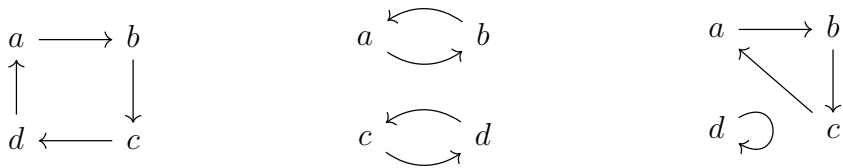
\begin{figure}[h]
    \begin{minipage}{0.225\textwidth}
        \begin{flushleft}
        	\begin{tikzcd}
                a \arrow[r] & b \arrow[d]\\
                d \arrow[u] & c \arrow[l]
            \end{tikzcd}
        \end{flushleft}
    \end{minipage}
    \begin{minipage}{0.225\textwidth}
        \begin{center}
        \begin{tikzcd}
            a \arrow[r, bend right = 35, swap] & b \arrow[l, bend right = 35, swap]\\
            c \arrow[r, bend right = 35, swap] & d \arrow[l, bend right = 35, swap]
    \end{tikzcd}
        \end{center}
    \end{minipage}
    \begin{minipage}{0.225\textwidth}
        \begin{flushright}
            \begin{tikzcd}
            a \arrow[r] & b \arrow[d]\\
            d \arrow[loop,out=30,in=-30,distance=20,""] & c \arrow[lu]
        \end{tikzcd}
        \end{flushright}
    \end{minipage}
    \caption{Examples of postsingular portraits of non-entire Thurston maps that satisfy conditions \ref{it: set A} and \ref{it: set B}.}
    \label{fig: non entire portraits}
    \end{figure}

    Of course, there are more examples, e.g., a Thurston map $f$ with $P_f = \{a, b, c, d\}$ satisfies condition \ref{it: set B} if $S_f = \{a, b, c\}$, $d$ is the essential singularity of the map $f$, $f(a) = d$, and $b$ and $c$ form a 2-cycle for the map $f$.
\end{example}

\begin{example}[Maps with three postsingular values]\label{ex: three postsingular values}
    Suppose that $f \colon (S^2, A) \rto$ is a Thurston map such that $|A| = 4$ and $|P_f| \leq 3$. It is easy to see that if there exists a marked point $a \in A - P_f$ that is not periodic (i.e., it is either pre-periodic or lands to the essential singularity of $f$ under the iteration), then the pullback map $\sigma_{f}$ is constant by Proposition \ref{prop: dependence} and the Thurston map $f \colon (S^2, A) \rto$ is realized, as noted in Remark \ref{rem: constant sigma map}.

    On the other hand, if every marked point $a \in A - P_f$ is periodic, then conditions~\ref{it: set A} and~\ref{it: set B} are clearly satisfied. For instance, if $|P_f| = 3$, we can simply take $B = P_f$.
\end{example}

\appendix

\section{Few facts about dynamics of meromorphic maps}

We require the following two results regarding the dynamics of meromorphic maps. Although these results are mostly folklore, we provide short proofs for the completeness.

\begin{lemma}\label{lemm: classification of fixed point}
    Every postsingularly finite meromorphic function has only finitely many superattracting periodic orbits, and all other periodic orbits are repelling.
\end{lemma}

\begin{proof}
    Let $g \colon \C \to \widehat{\C}$ be a postsingularly finite meromorphic function. It is evident that $g$ can have only a finite number of superattracting periodic orbits, as the points of each such orbit belong to the postsingular~set of~$g$.
    
    Now, consider a periodic point $z \in \C$ of $g$. If $z$ is attracting (but not superattracting), then according to \cite[Theorem 7]{Bergweiler}, the corresponding cycle of immediate attracting basins contains a singular value $w \in S_g$ that has an infinite orbit, leading to a contradiction. Therefore, if $z$ is in the Fatou set of $g$, then it is the center of a cycle of Siegel disks $U_1, U_2, \dots, U_k$, and postsingular values of~$g$ are dense in $\partial U_i$ for each $i = 1, 2, \dots, k$ \cite[Theorem 7]{Bergweiler}. This implies that the postsingular set of $g$ would be infinite. If $z$ is in the Julia set of $g$, then it is either a Cremer periodic point or it lies on the boundary of a cycle of parabolic basins \cite[Theorem 7.2]{Milnor}. In both cases, $z$ is an accumulation point of the postsingular values of $g$ (see \cite[Theorem 7]{Bergweiler} and \cite[Proposition 16]{Adam_Thesis}; see also \cite[Lemma~72]{Adam_Thesis}). Thus, $z$ is either a superattracting or repelling periodic point of the function~$g$.
\end{proof}

The next result that we require states that a transcendental meromorphic function of finite type has infinitely many repelling fixed points. This was established in the more general context of finite type maps in \cite[Proposition~14]{Adam_Thesis}. Furthermore, in the paper~\cite{fixed_points}, it was shown that the same is true for transcendental meromorphic functions of \textit{bounded type} (i.e., having bounded singular set) under the assumption the $\infty$ is a \textit{logarithmic singularity} of the considered function. In the following lemma, we show that this assumption can be removed and, in fact, the result holds for an arbitrary transcendental meromorphic function~of~bounded~type.

\begin{lemma}\label{lemm: inf many repelling fixed points}
    Every transcendental meromorphic function of bounded type has infinitely many repelling fixed points.
\end{lemma}

\begin{proof}
    Let $g \colon \C \to \widehat{\C}$ be a transcendental meromorphic function of bounded type and $D \subset \widehat{\C}$ be an open Jordan region containing $\infty$ such that $S_g \cap \overline{D} = \{\infty\} \cap \overline{D}$. If $f^{-1}(D)$ has a connected component that is unbounded (in $\C$), then the result follows directly from \cite{fixed_points}. Now, suppose that every connected component of $f^{-1}(D)$ is bounded. In this case, $f^{-1}(D)$ has infinitely many connected components, and all but finitely many of them are compactly contained in $D$. Let $U$ be one such component, i.e., $U$ is a connected component of $g^{-1}(D)$ such that $\overline{U} \subset D$. Let $z$ be a unique pole of $g$ in $U$ and let $d := \deg(g, z)$. Note that $U$ is an open Jordan region and $g|U - \{z\} \colon U - \{z\} \to D - \{\infty\}$ is~covering~map~of~degree~$d$.
    
    Consider a Jordan arc $\alpha \subset \overline{D}$ connecting $\infty$ with a point on $\partial D$, with the conditions that $|\alpha \cap \partial D| = 1$ and $\alpha \cap \overline{U} = \emptyset$. Then $g^{-1}(\alpha)$ subdivides $U$ in $d$ simply connected domains $U_1, U_2, \dots, U_d$. Furthermore, the restriction $g|U_i \colon U_i \to D - \alpha$ is a biholomorphism. Proposition \ref{prop: mcmullens bound} and Remark \ref{rem: mcmullens bound} imply that the inverse $(g|U_i)^{-1}$ is uniformly distance-decreasing with respect to the hyperbolic metric on $D - \alpha$, because $U_i = g(D - \alpha)$ is compactly contained in $D - \alpha$. Therefore, by the Banach fixed point theorem, $g$ has a fixed point in each $U_i$ for $i = 1, 2, \dots, d$. These fixed points are attracting for $(g|U_i)^{-1}$ and thus repelling for the map $g$. By applying the same argument to every connected component of $g^{-1}(D)$ that is compactly contained in $D$, we conclude that the map $g$ has infinitely many~repelling~fixed~points.
\end{proof}

\bibliographystyle{alpha}
\bibliography{lib.bib}

\end{document}